\documentclass[12pt, a4paper]{amsart}
\usepackage{amsmath}
\usepackage{geometry,amsthm,graphics,tabularx,amssymb,shapepar}
\usepackage{amscd}
\usepackage{graphicx}
\usepackage[all]{xypic}
\usepackage{hyperref}
\usepackage{color}
\usepackage{soul}
\usepackage{bm}
\usepackage{bbm}
\newcommand{\fm}{\mathfrak m}
\newcommand{\fn}{\mathfrak n}
\newcommand{\fk}{\mathfrak k}
\newcommand{\fg}{\mathfrak g}
\newcommand{\fh}{\mathfrak h}
\newcommand{\dg}{\dot{\mathfrak g}}
\newcommand{\dfh}{\dot{\mathfrak h}}
\newcommand{\dQ}{\dot{Q}}

\newcommand{\al}{\alpha}

\newcommand{\wt}{\widetilde}
\newcommand{\wh}{\widehat}

\newcommand{\dal}{\dot{\alpha}}

\newcommand{\ot}{\otimes}
\newcommand{\CK}{\mathcal{K}}

\newcommand{\CL}{\mathcal{L}}

\newcommand{\CE}{\mathcal{E}}

\newcommand{\C}{\mathbb{C}}
\newcommand{\R}{\mathbb R}
\newcommand{\N}{\mathbb N}
\newcommand{\Z}{\mathbb Z}
\newcommand{\rd}{\mathrm{d}}
\newcommand{\rk}{\mathrm{k}}

\newcommand{\ba}{\begin {eqnarray}}
\newcommand{\ea}{\end {eqnarray}}
\newcommand{\baa}{\begin {eqnarray*}}
\newcommand{\eaa}{\end {eqnarray*}}
\newcommand{\be}{\begin {equation}}
\newcommand{\ee}{\end {equation}}
\newcommand{\bee}{\begin {equation*}}
\newcommand{\eee}{\end {equation*}}

\newcommand{\U}{\mathcal{U}}
\newcommand{\belowit}[2]{{
    \mathop{{#1}}\limits_{{#2}}
}}

\newcommand{\te}[1]{\textnormal{{#1}}}

\theoremstyle{Theorem}

\theoremstyle{Theorem}

\newtheorem{thm}{Theorem}[section]
\newtheorem{cort}[thm]{Corollary}
\newtheorem{lemt}[thm]{Lemma}
\newtheorem{prpt}[thm]{Proposition}

\newtheorem{remt}[thm]{Remark}

\theoremstyle{Theorem}

\theoremstyle{Theorem}

\theoremstyle{Plain}

\theoremstyle{Definition}

\newtheorem{dfnt}[thm]{Definition}

\def\({\left(}

\def\){\right)}

\newlength{\dhatheight}
\newcommand{\dwidehat}[1]{%
    \settoheight{\dhatheight}{\ensuremath{\widehat{#1}}}%
    \addtolength{\dhatheight}{-0.45ex}%
    \widehat{\vphantom{\rule{1pt}{\dhatheight}}%
    \smash{\widehat{#1}}}}

\newcommand{\set}[2]{{
    \left.\left\{
        {#1}
    \,\right|\,
        {#2}
    \right\}
}}

\def \<{{\langle}}
\def \>{{\rangle}}

\newcommand{\vac}{\mathbbm{1}}
\newcommand{\E}{{\mathcal{E}}}

\newcommand{\CY}{\mathcal{Y}}

\newcommand{\ssl}{\mathfrak{sl}}

\allowdisplaybreaks

\newcommand{\inverse}{^{-1}}

\numberwithin{equation}{section}
\title[Drinfeld type presentations]{Drinfeld type presentations of loop algebras}

\author{Fulin Chen$^1$}
\address{School of Mathematical Sciences, Xiamen University,
 Xiamen, China 361005} \email{chenf@xmu.edu.cn }\thanks{$^1$Partially supported by NSF of China (No.11501478).}

\author{Naihuan Jing$^2$}
\address{Department of Mathematics, North Carolina State University, Raleigh, NC 27695,
USA}
\email{jing@math.ncsu.edu}
\thanks{$^2$Partially supported by NSF of China (No.11531004, No.11726016) and Simons Foundation (No.523868).}

\author{Fei Kong$^3$}
\address{College of Mathematics and Statistics, Hunan Normal University, Changsha, China 410006} \email{kfkfkfc@outlook.com}
\thanks{$^3$Partially supported by NSF of China (No.11701183).}

 \author{Shaobin Tan$^4$}
 \address{School of Mathematical Sciences, Xiamen University,
 Xiamen, China 361005} \email{tans@xmu.edu.cn}
 \thanks{$^4$Partially supported by NSF of China (No.11471268, No.11531004).}


\begin{document}

\begin{abstract}Let $\fg$ be the derived subalgebra of a Kac-Moody Lie algebra of finite type or  affine type,
$\mu$ a diagram automorphism of $\fg$ and $\CL(\fg,\mu)$ the loop algebra
 of $\fg$ associated to $\mu$.
In this paper, by using the vertex algebra technique, we provide a general construction of
current type  presentations for the universal central extension  $\wh\fg[\mu]$ of $\CL(\fg,\mu)$.
 The construction contains the
classical limit of  Drinfeld's new realization for (twisted and untwisted)
  quantum affine algebras (\cite{Dr-new}) and
the Moody-Rao-Yokonuma
presentation for toroidal Lie algebras (\cite{MRY}) as special examples.
As an application, when $\fg$ is of simply-laced  type, we  prove that the classical limit of
 the $\mu$-twisted quantum
 affinization  of the quantum Kac-Moody algebra associated to $\fg$ introduced in \cite{CJKT-twisted-quantum-aff-vr} is
 the universal enveloping algebra of $\wh\fg[\mu]$.
\end{abstract}
\maketitle

\section{Introduction and main results}
\subsection{The main result}
 Let $A=(a_{ij})_{i,j\in I}$ be a  generalized Cartan matrix of \emph{finite} type or \emph{affine }type,
  $\mu$ a permutation of $I$ with order $N$ such that $a_{\mu(i)\mu(j)}=a_{ij}$ for $i,j\in I$,
 and $\fg$ the derived subalgebra of the Kac-Moody Lie algebra associated to $A$.
 It was known (\cite{GK-defining-rel-Lie-alg}) that  $\fg$ is  generated by the
Chevalley generators $\al_i^\vee,\, e_i^\pm$, $i\in I$ and subject to the  relations
 \begin{align}
&[\al_i^\vee,\al_j^\vee]=0,\ [\al_i^\vee, e_j^\pm]=\pm a_{ij}\,e_j^\pm,\
[e_i^+,e_j^-]=\delta_{ij}\,\al_i^\vee,\quad i,j\in I,\\
\label{Serrere}
&\mathrm{ad}(e_i^\pm)^{1-a_{ij}}(e_j^\pm)=0,\quad i,j\in I\ \te{with}\ i\ne j.
 \end{align}
This presentation, known as the Serre-Gabber-Kac presentation
of $\fg$, is of fundamental importance in the study of the Kac-Moody Lie algebra. 
Let $\fn_+$ (resp.\,$\fn_-$) be   the subalgebra of $\fg$ generated by the elements $e_i^+$ (resp.\,$e_i^-$)
 for $i\in I$.
One of the advantages of this presentation is that the subalgebras $\fn_+$ and $\fn_-$  are abstractly generated by
 these elements with the  Serre
relations \eqref{Serrere}.
The Serre-Gabber-Kac presentation also implies that $\mu$ induces an automorphism of $\fg$,
still denoted as $\mu$ (called diagram automorphism), such that
\begin{align}\label{defmu}\mu(\al_i^\vee)=\al_{\mu(i)}^\vee,\quad
\mu(e_i^\pm)=e_{\mu(i)}^\pm,\quad i\in I.
\end{align}

Let $\wh\fg$ be the universal central extension of the loop algebra $\mathcal L(\fg)=\C[t_1,t_1^{-1}]\ot \fg$.
 When $A$ is of finite type, $\wh\fg=\CL(\fg)\oplus\C\rk_1$ is an untwisted affine Lie algebra;
  when $A$ is of  affine type, $\wh\fg$ is a toroidal Lie algebra with
 $\CL(\fg)\oplus \C\rk_1$ as a (proper) subspace.
Similar to \eqref{defmu},  $\mu$ also induces an automorphism $\wh\mu$ of $\wh\fg$ such that (see \S\,3.1)
\begin{align}\label{defhatmu0}\wh\mu(t_1^m\ot \al_i^\vee)=\xi^{-m}t_1^m\ot \al_{\mu(i)}^\vee,\
\wh\mu(t_1^m\ot e_i^\pm)=\xi^{-m}t_1^m\ot e_{\mu(i)}^\pm,\ \wh\mu(\rk_1)=\rk_1,
\end{align}
for $m\in \Z$, $i\in I$ and $\xi=\mathrm{e}^{2\pi\sqrt{-1}/N}$.
Let $\wh\fg[\mu]$ be the subalgebra of $\wh\fg$ fixed by $\wh\mu$, which is the Lie algebra concerned about in this paper.
It was known  (\cite{Kac-book,CJKT-uce})  that $\wh\fg[\mu]$ is the universal central extension of
the (twisted) loop algebra
\begin{align}\label{fm}
\CL(\fg,\mu)=\te{Span}_\C\{ t_1^{m}\ot x_{(m)}\mid x\in \fg,\, m\in \Z\}\subset \CL(\fg)
\end{align}
of $\fg$ associated to $\mu$, where $x_{(m)}=\sum_{k\in \Z_N}\xi^{-km}\mu^k(x)$ and
 $\Z_N=\Z/N\Z$.
The main goal of this paper is to provide a general construction of certain current type presentations   for $\wh\fg[\mu]$,  which are  loop analogues of the Serre-Gabber-Kac presentation for $\fg$.

To be more precise,
let $D=\te{diag}\{\epsilon_i, i\in I\}$ be a diagonal matrix of positive rational numbers such that  $DA$ is symmetric,
and let $\wh\fn_+[\mu]$ (resp.\,$\wh\fn_-[\mu]$)  be the subalgebra of $\wh\fg[\mu]$  generated by the elements
$t_1^m\ot e_{i(m)}^+$ (resp.\,$t_1^m\ot e_{i(m)}^-$)  for $i\in I$, $m\in \Z$.
We now recall a result of Drinfeld as motivation.

\begin{thm}\label{intro1} Assume that $A$ is of finite type.
 Then the affine Lie algebra $\wh\fg[\mu]$ is isomorphic to the Lie algebra
generated by the elements
 \begin{align}\label{generators} h_{i,m},\quad x_{i,m}^\pm,\quad c, \quad i\in I,\ m\in \Z \end{align}
 and
 subject to the relations ($i,j\in I, m,n\in \Z$)
\begin{eqnarray*}
&&\text{(H)}\quad h_{\mu(i),m}=\xi^m h_{i,m},\ [h_{i,m}, c]=0,\ [h_{i,m}, h_{j,n}]=\sum_{k\in \Z_N}\frac{mN}{\epsilon_j} a_{i\mu^k(j)}\xi^{km}\delta_{m+n,0}c,\\
&&\text{(HX\,$\pm$)}\quad [h_{i,m}, x^\pm_{j,n}]
=\pm \sum_{k\in \Z_N} a_{i\mu^k(j)}\xi^{km}x^\pm_{j,m+n},\ [x^\pm_{j,n},c]=0,\\
&&\text{(XX)}\quad [x^+_{i,m},x^-_{j,n}]
=\sum_{k\in \Z_N}\delta_{i,\mu^k (j)}\xi^{km}(h_{j,m+n}+\frac{mN}{\epsilon_j}\delta_{m+n,0}c ),\\
&&\text{(X\,$\pm$) }\quad x^\pm_{\mu(i),m}=\xi^{m}x^\pm_{i,m},\ f_{ij}(z,w)\cdot[x_i^\pm(z),x_j^\pm(w)]=0,\\
&&\text{(DS\,$\pm$) }\quad
[x_i^\pm(z_{1}),\cdots,[x_i^\pm(z_{1-a_{ij}}), x_j^\pm(w)]] =0,\ \te{if}\ a_{ij}<0\ \te{and}\ \mu(i)=i;\\
&&\qquad \quad [x_i^\pm(z_{1}),[x_i^\pm(z_2), x_j^\pm(w)]] =0,\ \te{if}\ a_{ij}=-1,  a_{i\mu(i)}=0\ \te{and}\ \mu(j)\ne j;\\
&&\ \,
\frac{z^N-1}{z-1}\,[x_i^\pm(z_{1}),[x_i^\pm(z_2), x_j^\pm(w)]] =0,\ \te{if}\ a_{ij}=-1,  a_{i\mu(i)}=0\ \te{and}\ \mu(j)=j;\\
&&\ \,
(z+1)\,[x_i^\pm(z_{1}),[x_i^\pm(z_2), x_j^\pm(w)]] =0,\ \te{if}\ a_{ij}=-1, a_{i\mu(i)}=-1\  \te{and}\ j\ne \mu(i);\\
&&\,\sum_{\sigma\in S_2}\(z_{\sigma(1)}-2z_{\sigma(2)}-w\)
[x_i^\pm(z_{\sigma(1)}),[x_i^\pm(z_{\sigma(2)}),x_j^\pm(w)]]=0,\ \te{if}\ a_{ij}=-1, j=\mu(i),
\end{eqnarray*}
where $f_{ij}(z,w)=\prod\limits_{k\in\Z_N;a_{i\mu^kj}\ne 0} (z-\xi^k w)$ and  $x_i^\pm(z)=\sum_{m\in \Z}x_{i,m}^\pm z^{-m-1}$.
The isomorphism with $\wh\fg[\mu]$  is induced by the assignment
\begin{align}\label{thetafg} h_{i,m}\mapsto t_1^m\ot \al_{i(m)}^\vee,\ x_{i,m}^\pm\mapsto t_1^m\ot e_{i(m)}^\pm,\ c\mapsto \rk_1,\quad
i\in I,\, m\in \Z.
\end{align}
Moreover,
$\wh\fn_+[\mu]$ (resp.\,$\wh\fn_-[\mu]$) is isomorphic to
the Lie algebra  generated by $x^+_{i,m}$, (resp.\,$x^-_{i,m}$) for $i\in I$, $m\in \Z$ with relations $(X+)$, $(DS+)$
(resp.\,$(X-)$, $(DS-)$).
The isomorphism with $\wh\fn_+[\mu]$ (resp.\,$\wh\fn_-[\mu]$) is induced by the assignment
\begin{align}\label{thetafn}\ x_{i,m}^+\mapsto t_1^m\ot e_{i(m)}^+,\quad (\te{resp}.\ x_{i,m}^-\mapsto t_1^m\ot e_{i(m)}^-),\quad
i\in I,\, m\in \Z.
\end{align}
\end{thm}

The current algebra presentation of  $\wh\fg[\mu]$ given in Theorem \ref{intro1}
is the classical limit of  the Drinfeld's new realization for (twisted and untwisted) quantum affine algebras (\cite{Dr-new}),
which plays a key role in understanding the  isomorphism between
Drinfeld-Jimbo's definition and Drinfeld's new realization for quantum affine algebras (\cite{beck,Da1,Da2,JZ}).
When $\mu=\mathrm{Id}$, a version of Theorem \ref{intro1} was given in \cite{Gar-loop-alg}.
For the general case, a proof of Theorem \ref{intro1} was given in \cite{Da2}.

Motivated by the Serre relations ($DS\pm$) given in Theorem \ref{intro1}, we are interested in
those Drinfeld type Serre relations  in $\wh\fg[\mu]$
which have the form
\begin{align*}
(\bm{P}1)\ \sum_{\sigma\in S_{1-a_{ij}}}P_{ij,\sigma}(z_1,\cdots,z_{1-a_{ij}},w)\cdot [e_i^\pm(z_{\sigma(1)}),\cdots, [e_i^\pm(z_{\sigma(1-a_{ij})}), e_j^\pm(w)]]=0,
 \end{align*}
 where $(i,j)\in \mathbb{I}=\{(i,j)\in I\times I\mid a_{ij}<0\},$  $e_i^\pm(z)=\sum_{m\in \Z} t_1^m\ot e_{i(m)}^\pm z^{-m-1}$
 and $P^\pm_{ij,\sigma}(z_1,\cdots,z_{1-a_{ij}},w)$
  are some homogenous polynomials.
Starting with any such Serre relations in $\wh\fg[\mu]$, we introduce the following definition.

  \begin{dfnt} \label{defDPgmu}Let
\begin{align}\label{P}\bm{P}=\{P_{ij,\sigma}(z_1,\cdots,z_{1-a_{ij}},w)\mid (i,j)\in \mathbb{I},\  \sigma\in S_{1-a_{ij}}\}\end{align}
be a family of homogenous  polynomials which satisfies the condition ($\bm{P}$1).
We define $\mathcal D_{\bm{P}}(\fg,\mu)$ to be the Lie algebra generated by the elements as in \eqref{generators}
and subject to the relations $(H), (HX\pm), (XX), (X\pm)$ as in Theorem \ref{intro1}
together with the following Serre relations
$((i,j)\in \mathbb I)$
 \begin{align*}
(AS\pm)&\quad (z_1^N-z_2^N)\cdot [x_i^\pm(z_1),[x_i^\pm(z_2),x_j^\pm(w)]]=0,\quad\te{if $\fg$ is of type $A_1^{(1)}$},\\
(DS\pm)_{\bm{P}}&\ \sum_{\sigma\in S_{1-a_{ij}}}P_{ij,\sigma}(z_1,\cdots,z_{1-a_{ij}},w)\, [x_i^\pm(z_{\sigma(1)}),\cdots, [x_i^\pm(x_{\sigma(1-a_{ij})}), x_j^\pm(w)]]=0,
 \end{align*}
where the polynomial $f_{ij}(z,w)$ in $(X\pm)$ is now given by
\begin{equation*}
f_{ij}(z,w)=\begin{cases}\prod\limits_{k\in \Z_N;a_{i\mu^k(j)\ne 0}}(z-\xi^k w),\ &\te{if $A$ is not of type $A_1^{(1)}$};\\
(z-w)\cdot \prod\limits_{k\in \Z_N;a_{i\mu^k(j)< 0}}(z-\xi^k w)^2,\ &\te{if $A$ is  of type $A_1^{(1)}$}.
\end{cases}
\end{equation*}

 Similarly, we define $\mathcal D_{\bm{P}}(\fn_+,\mu)$ (resp.\,$\mathcal D_{\bm{P}}(\fn_-,\mu)$)
 to be the Lie algebra generated by the elements $x_{i,m}^+$ (resp.\,$x_{i,m}^-$) for $i\in I, m\in \Z$
and subject to the relations  $(X+)$, $(AS+)$, $(DS+)_{\bm{P}}$ (resp.\,$(X-)$, $(AS-)$, $(DS-)_{\bm{P}}$).
\end{dfnt}

  By definition, for $\fm=\fg$ or $\fn_+, \fn_-$, the assignment \eqref{thetafg} or \eqref{thetafn}
determines a surjective Lie homomorphism, say $\theta_{\fm,\bm{P}}:\mathcal D_{\bm{P}}(\fm,\mu)\rightarrow \wh\fm[\mu]$ (see Lemma \ref{thetahom}).

\begin{dfnt}
We say that the Lie algebra $\mathcal D_{\bm{P}}(\fg,\mu)$ is a Drinfeld type presentation of $\wh\fg[\mu]$ if
for every $\fm=\fg,\fn_+$ or $\fn_-$, the Lie homomorphism $\theta_{\fm,\bm{P}}$
is an isomorphism.
\end{dfnt}

In view of Theorem \ref{intro1}, it is a natural question that for which suitable  $\bm{P}$
  the Lie algebra $\mathcal D_{\bm{P}}(\fg,\mu)$ is a Drinfeld type presentation of $\wh\fg[\mu]$ (especially when $\fg$ is of affine type)?
In this paper we give a surprising  answer to this   question
by pointing out that  the family $\bm{P}$ only need to satisfy the following natural condition
\begin{align*}
(\bm{P}2)\quad \sum_{\sigma\in S_{1-a_{ij}}}P_{ij,\sigma}(w,\cdots,w,w)\ne 0,\quad \forall\ (i,j)\in \mathbb I.\qquad \end{align*}
Indeed, as the main result of our paper, we prove that
\begin{thm}\label{intromain1} Assume that the family $\bm{P}$ satisfies the condition ($\bm{P}2$).
Then the Lie algebra $\mathcal D_{\bm{P}}(\fg,\mu)$ is a Drinfeld type presentation of $\wh\fg[\mu]$.
\end{thm}

We say that a module $W$ of a $\Z$-graded Lie algebra  $\mathfrak p=\oplus_{n\in \Z}\mathfrak p_n$
is  {\em restricted} if for every $w\in W$, there exists a
 positive integer $s$ such that $\mathfrak p_{n}.w=0$ for all $n\ge s$.
Note that there is a natural $\Z$-grading structure on the algebra $\mathcal D_{\bm{P}}(\fg,\mu)$
(resp.\,$\mathcal D_{\bm{P}}(\fn_+,\mu)$ or $\mathcal D_{\bm{P}}(\fn_-,\mu)$) with ($i\in I, m\in \Z$)
\begin{align}
\deg(x_{i,m}^\pm)=m=\deg(h_{i,m}),\ (\te{resp.}\,\deg(x_{i,m}^+)=m\ \te{or}\
\deg(x_{i,m}^-)=m).
\end{align}
By a standard result on restricted modules of $\Z$-graded Lie algebras (see Lemma \ref{lem:injective}),
the proof of Theorem \ref{intromain1} can be reduced to the following theorem.
\begin{thm} \label{intromain2}  Assume that the family $\bm{P}$ satisfies the condition ($\bm{P}2$).
Then for any restricted $\mathcal D_{\bm{P}}(\fm,\mu)$-module $W$, where $\fm=\fg,\fn_+$ or $\fn_-$,
there exists an  $\wh\fm[\mu]$-module structure on $W$
 such that
\begin{align} x.w=\theta_{\fm,\bm{P}}(x).w,\quad x\in \mathcal D_{\bm{P}}(\fm,\mu),\ w\in W.
\end{align}
\end{thm}

The main body of this paper is devoted to a proof of Theorem \ref{intromain2}, which is based on
the theory of $\Gamma$-vertex algebras and their quasi-modules developed by Li in \cite{Li-new-construction,Li-Gamma-quasi-mod}.
In \S\,3.3, we construct a family $\bm{p}$ of polynomials which satisfies the conditions
($\bm{P}1$) and $(\bm{P}2)$.
When $\fg$ is of finite type, the  relations in $(DS\pm)_{\bm{p}}$ are the same as that in $(DS\pm)$   except
the last one, which is much simpler:
\[[x^\pm_i(z_1),[x^\pm_i(z_2),x^\pm_j(w)]]=0,\quad \te{if}\quad a_{ij}=-1,\, j=\mu(i).\]
In particular, this shows that  Theorem \ref{intro1} is just a special case of Theorem \ref{intromain1}.
 When $\fg$ is of untwisted affine type and $\mu=\mathrm{Id}$, the presentation $\mathcal D_{\bm{p}}(\fg,\mu)$ of the toroidal Lie algebra $\wh\fg$
 was first introduced by
  Moody-Rao-Yokonuma  in \cite{MRY} and was often called  the MRY presentation of $\wh\fg$.
  When $\fg$ is not of type $A_1^{(1)}$ and $\mu=\mathrm{Id}$, the presentation $\mathcal D_{\bm{p}}(\fg,\fn_+)$ for $\wh\fn_+$ was first proved in
\cite{E-PBW-qaff} for the purpose of understanding the classical limit of quantum current algebras associated to $\fg$.

When $A$ is of affine type and $\mu$ is non-transitive, it was known (\cite{FSS}) that the $\mu$-folded matrix
$\check{A}=(\check{a}_{ij})$ associated to $A$ is also an affine generalized Cartan matrix.
In this case,  we  constructed in \cite{CJKT-uce} a current type presentation
for $\wh\fg[\mu]$ with a different  Serre relation:
\begin{align}\label{eq:DJ-Serre-rel-intro}
[x_i^\pm(z_1),\cdots,[x_i^\pm(z_{1-\check{a}_{ij}}),x_j^\pm(w)]]=0,\quad \te{if}\ \check{a}_{ij}<0.
\end{align}
However, as pointed out in \cite{Da2}, when $\mu$ is non-trivial, one can not obtain a presentation for $\wh{\fn}_\pm[\mu]$
 by replacing the  Serre relation $(DS\pm)_{\bm{P}}$  with
\eqref{eq:DJ-Serre-rel-intro}.

\subsection{The main motivation}
The main motivation of this paper stems from the quantization theory of extended affine Lie algebras (EALAs for short).
The notion of EALAs  was first introduced by H{\o}egh-Krohn and Torresani (\cite{H-KT}), and
the theory of EALAs has been intensively studied for over twenty-five years
 (see \cite{AABGP,BGK,N2}
and the references therein).
 An EALA $E$ is by definition a complex Lie algebra, together with a Cartan subalgebra and a non-degenerate invariant symmetric bilinear form,
that satisfies
a list of natural axioms.
The invariant form on  $E$ induces a semi-positive bilinear form on the $\R$-span of the root system $\Phi$ of $E$,
and so
 $\Phi$ divides into a disjoint union of the sets of isotropic and non-isotropic roots.
 Roughly speaking, the structure of an EALA is determined by its core,
   the subalgebra  generated by non-isotropic root vectors.

 One of the axioms for EALAs requires that the rank of the
group generated by the isotropic roots is finite, and this rank is called the nullity of EALAs.
The nullity $0$ EALAs are nothing but the finite dimensional simple Lie algebras,
the nullity $1$  EALAs are precisely the
 affine Kac-Moody algebras  \cite{ABGP}, and
the nullity $2$ EALAs are
closely related to the Lie algebras studied by Saito and Slodowy in the work of singularity theory.
Recently, a complete classification of the
centerless cores (the core modulo its center) of nullity $2$ EALAs was obtained in \cite{ABP} (see also \cite{GP-torsors}).
Let $\wh{E}_n$ denote
 the class of Lie algebras which are isomorphic to the core of some EALA with nullity $n$ and are central closed.
For convenience of the readers, we ``describe" in the following Figure the relations between  EALAs and the algebras
concerned about in this paper:
\begin{align*}&\xymatrix{*+[F]{\wh{E}_0}}=\xymatrix{*+[F]{\te{finite type}\ \fg}}\, ,\quad
  \xymatrix{*+[F]{\wh{E}_1}}= \xymatrix{*+[F]{\te{affine type}\ \fg}}=\xymatrix{*+[F]{\wh\fg[\mu],\ \fg:\te{\,finite type}}}\, ,\\
 &\xymatrix{*+[F]{\wh{E}_2}}=\xymatrix{*+[F]{\wh\fg[\mu],\ \fg: \te{affine},\
                \mu: \te{non-transitive}}}+\xymatrix{*+[F]{\wh\ssl_\ell(\C_p),\ \ell\ge 2,\ p: \te{generic}}}\, .
  \end{align*}
Here, the notation $\wh\ssl_\ell(\C_p)$ stands for the universal central extension of the special linear Lie algebra over
the quantum torus associated to $p\in \C^\times$ (\cite{BGK}).

Let $\U_{\hbar}(\fg)$ be the quantum Kac-Moody algebra over $\C[[\hbar]]$ associated to $\fg$, whose theory
 has been a tremendous success story.
Similar to the classical case, quantum affine algebras also  have two realizations: Drinfeld-Jimbo's original
realization and Drinfeld's new realization as  quantum affinizations  of  quantum finite algebras.
By applying Drinfeld's untwisted quantum affinization process to a quantum affine algebra $\U_{\hbar}(\fg)$, one
obtains a quantum toroidal algebra $\U_{\hbar}(\wh\fg)$ (\cite{GKV,J-KM,Naka-quiver,He-representation-coprod-proof}).
In the particular case  $\fg=\wh\ssl_{\ell+1}$, an additional parameter $p$ can be added in this quantum affinization
process \cite{GKV} and then one gets a two parameter deformed algebra $\mathcal U_{\hbar,p}(\dwidehat{{\mathfrak{sl}}}_{\ell+1})$, which
is also called a quantum toroidal algebra.
The theory of quantum toroidal algebras has been intensively studied since their discovery.
In particular,  the representation theory of quantum toroidal algebras is very rich and promising (see \cite{He-total} for a survey).
Let $\U(\wh{E}_n)$ denote the class of  universal enveloping algebras of the Lie algebras in $\wh{E}_n$.
It is important for us to observe that
 all the quantum algebras mentioned above are related to EALAs by taking the classical limits:
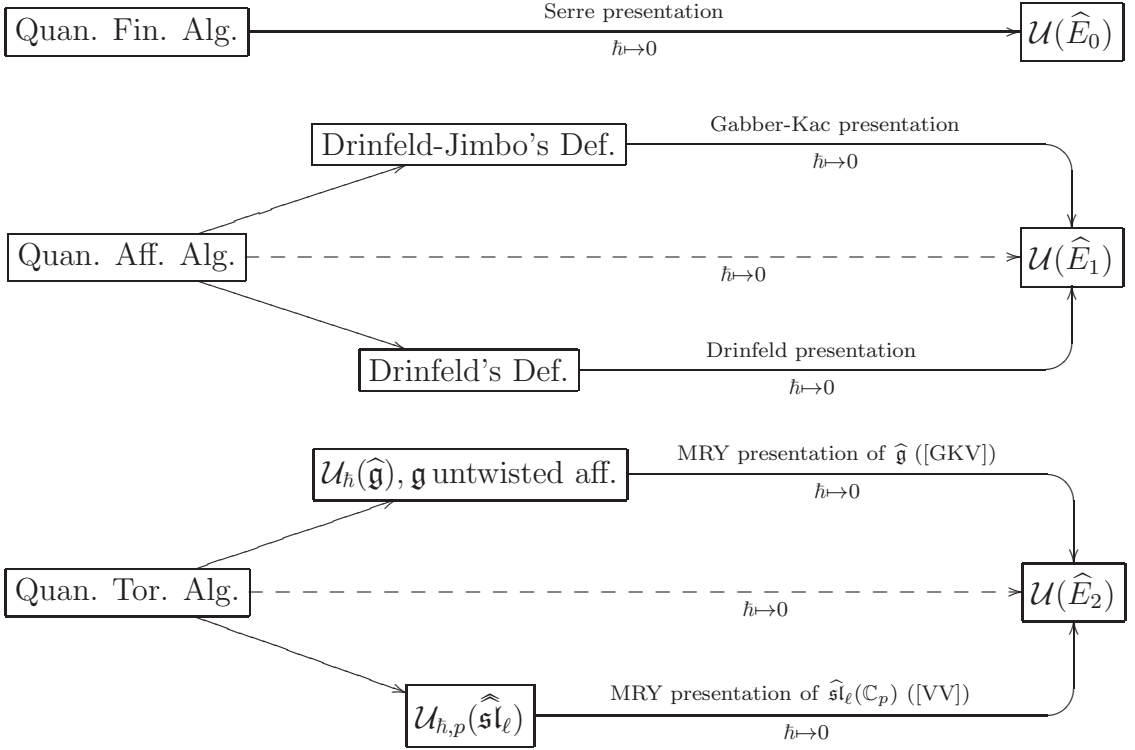
\begin{align*}\xymatrix{*+[F]{\te{Quan. Fin. Alg.} }\ar[rrrrrr]_-{\hbar\mapsto 0}^-{\te{Serre presentation}}&&&&&&*+[F]{\U(\wh{E}_0)}\\
   &*+[F]{\te{Drinfeld-Jimbo's Def.}}\ar `r[rrrrr]_-{\hbar\mapsto 0}^-{\te{Gabber-Kac presentation}} [rrrrrd]&&&&&
   \\
 *+[F]{\te{Quan. Aff. Alg.}} \ar[ur] \ar[dr]\ar@{-->}[rrrrrr]_-{\qquad\qquad\qquad\quad
 \hbar\mapsto 0}    & & &&&&  *+[F]{\U(\wh{E}_1)}    \\
   &*+[F]{\te{Drinfeld's Def.}} \ar `r[rrrrru]_-{\hbar\mapsto 0}^-{\te{Drinfeld presentation}}  [rrrrru]& &&&&}
\end{align*}
\begin{align*}
\xymatrix{
   &*+[F]{\U_{\hbar}(\wh\fg), \fg\, \te{untwisted aff.}} \ar `r[rrrrr]_-{\hbar\mapsto 0}^-{{\te{MRY presentation of}\ \wh\fg}\te{ (\cite{GKV})}} [rrrrrd] & &&&&
     \\
 *+[F]{\te{Quan. Tor. Alg.}} \ar[ur] \ar[dr]\ar@{-->}[rrrrrr]_-{\qquad\qquad\qquad\qquad\,\, \hbar\mapsto 0}  &&&&   & & *+[F]{\U(\wh{E}_2)}      \\
   &*+[F]{\U_{\hbar,p}(\dwidehat{\mathfrak{sl}}_\ell)} \ar `r[rrrrru]^-{{\te{MRY presentation of}\ \wh\ssl_\ell(\C_p)}\te{ (\cite{VV-double-loop})}}_-{\quad \hbar\mapsto 0} [rrrrru]&&&&&       }
\end{align*}

One of the most fundamental  problem in the theory of EALAs is that: just like the quantum finite, affine and toroidal algebras,
what is the ``right" $\hbar$-deformation of
the algebras in $\wh{E}_n$ for the general nullity $n$?
The theory of quantum toroidal algebras suggests that the nullity $2$ case is of particular interesting.
In this case, due to the classification result obtained in \cite{ABP}, it suffices to consider
 the $\hbar$-deformation of  $\wh{\fg}[\mu]$, where $\fg$ is of affine type and $\mu$ is non-transitive.
Motivated by Drinfeld's realization for twisted quantum affine algebras and the definition of quantum toroidal algebras,
this can be achieved via the following two steps:
(i). define the $\mu$-twisted quantum affinization $\U_\hbar(\wh\fg_\mu)$ of the quantum affine algebra $\U_\hbar(\fg)$; (ii).
prove that the classical limit of $\U_\hbar(\wh\fg_\mu)$ is isomorphic to $\U(\wh{\fg}[\mu])$.

The second Figure given above suggests us that the classical limit of $\U_\hbar(\wh\fg_\mu)$ should be
certain Drinfeld type presentations of $\wh\fg[\mu]$ constructed in Theorem \ref{intromain1}.
Indeed,
in \cite{CJKT-twisted-quantum-aff-vr}, by the quantum vertex operators technique,
we construct a new quantum algebra $\U_\hbar(\wh\fg_\mu)$ when $\fg$ is of simply-laced affine type and $\mu$ is non-transitive.
Due to Theorem \ref{intromain1}, we prove in \S\,6 that  the classical limit of
this quantum algebra is  isomorphic to $\U(\wh\fg[\mu])$. The twisted quantum affinization of
non-simply-laced quantum affine algebras will be defined in a forthcoming work, which together with Theorem \ref{intromain1}
 gives a complete answer to the previous fundamental  problem on the quantization theory of nullity $2$ EALAs.


\subsection{The structure of our proof}
Our original goal of this paper is to establish a suitable Drinfeld type presentation $\mathcal D_{\bm{P}}(\fg,\mu)$
for $\wh\fg[\mu]$ when $\fg$ is of affine type.
However,
comparing to the finite type case, there are many new  phenomenons appear
 in the affine type case:
(i) there exist non-trivial diagram automorphisms on non-simply-laced affine Cartan matrices (even in the untwisted case);
(ii) when $\mu$ is transitive, the $\mu$-folded matrix of $A$ is a zero matrix;
(iii) the algebra $\wh\fg[\mu]$ is no longer a Kac-Moody algebra and so does not admit a standard triangular decomposition;
(iv) $\wh\fn_\pm[\mu]$ is not the loop algebra of $\fn_\pm$ associated to $\mu$ but an infinite dimensional central extension of it.
These difficulties make  the method developed in \cite{Da2}  cannot be applied directly to the affine $\fg$
 even in the simplest case that $\bm{P}=\bm{p}$.
In this paper we employ the powerful vertex algebra approach to overcome these troubles.
Moreover, this new method allows us  to construct the
Drinfeld type presentations of $\wh\fg[\mu]$ in a very general setting.

We now outline the structure of our proof for Theorem \ref{intromain2}. Firstly,
we  construct  a ``universal" $\Gamma$-vertex algebra $V(\wh\fm)$
associated to $\wh\fm$ and  prove
 that the category of restricted modules for $\wh\fm[\mu]$ is equivalent to
the category of quasi-modules for $V(\wh\fm)$.
The main tool of our proof is the theory of  $\Gamma$-conformal Lie algebras developed in \cite{GKK-Gamma-conformal-alg,Li-new-construction}.
Next, due to Li's local system theory,   the locality relation ($X\pm$) implies that for any restricted
$\mathcal D_{\bm{P}}(\fm,\mu)$-module $W$,
there is an abstractly defined $\Gamma$-vertex algebra $\<U_{W,\fm}\>_{\Gamma}$
with $W$ as a quasi-module.
By analyzing the structure of this $\Gamma$-vertex algebra, we prove
that there is an $\wh\fm$-module structure on it,
in where the condition ($\bm{P}2$) is used.
Finally, using this fact we establish
a $\Gamma$-vertex homomorphism $\varphi_{\fm,W}$ from $V(\wh\fm)$ to  $\<U_{W,\fm}\>_{\Gamma}$.
Via this homomorphism, $W$ becomes a quasi-$V(\wh\fm)$-module and hence a $\wh\fm[\mu]$-module.
For convenience we ``describe" this process in the following Figure.
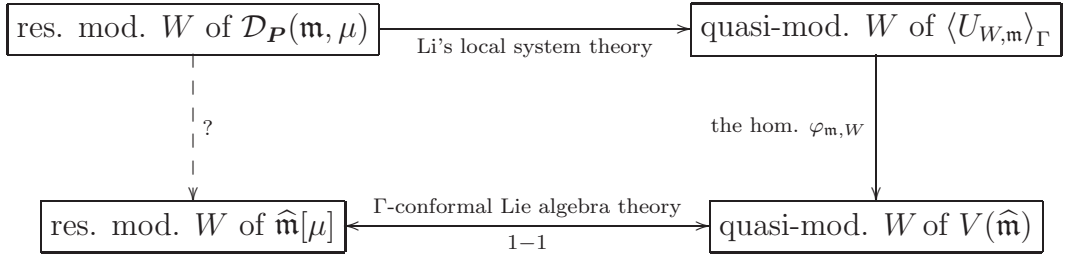
\begin{align*}
\xymatrix{
    *+[F]{\te{res. mod. $W$ of }\mathcal D_{\bm{P}}(\fm,\mu)}
    \ar[rrrr]_-{\te{Li's local system theory}}\ar@{-->}[dd]^?
        &&&&
    *+[F]{\te{quasi-mod. $W$ of }\<U_{W,\fm}\>_\Gamma}
    \ar[dd]_{\te{the hom. }\varphi_{\fm,W}}
    \\ \\
    *+[F]{ \te{res. mod. $W$ of }\wh\fm[\mu]} \ar[rrrr]_-{1-1}^-{\Gamma\te{-conformal Lie algebra theory}}
        &&&&
    *+[F]{\te{quasi-mod. }W\,\te{of }V(\wh\fm)}\ar[llll]
}
\end{align*}

According to the motivation explained in \S\,1.2, in this paper it is natural to simply assume that $\fg$ is of affine type.
However, since Theorem \ref{intromain1} is also new  when $\fg$ is of finite type and
virtually no extra work results, we choose to work in both finite and affine types.

This paper is organized as follows. In \S\,2, we
review some basics on Kac-Moody algebras and toroidal Lie algebras.
In \S\,3, we introduce the definition of the algebra $\wh\fg[\mu]$
and establish a class of Drinfeld type Serre relations in $\wh\fg[\mu]$.
As an application of Theorem \ref{intromain1}, the classical limit of the twisted quantum affinization algebra introduced in
\cite{CJKT-twisted-quantum-aff-vr} is determined in \S\,4.
\S\,6 is devoted to the proofs of Theorem \ref{intromain1} and Theorem \ref{intromain2},
and the basics on $\Gamma$-vertex algebras and their quasi-modules that are needed in the proof of Theorem \ref{intromain2} are collected in \S\,5.

All the Lie algebras considered in this paper are over the field $\C$ of complex numbers.
We denote the group of non-zero complex numbers, the set of non-zero integers, the set of positive integers and the set of non-negative integers  by $\C^\times$, $\Z^\times$, $\Z_+$ and $\N$, respectively.
For any $M\in \Z_+$, we set $\xi_M=\te{e}^{2\pi\sqrt{-1}/M}$.
And, for a Lie algebra $\fk$, we will use the notation $\U(\fk)$ to stand for the universal enveloping algebra of $\fk$.

\section{Preliminaries}
In this section, we fix notations and review some basics on Kac-Moody algebras and toroidal Lie algebras.
\subsection{Basics on Kac-Moody algebras}
Here we establish the notations we will use for finite dimensional simple Lie algebras and affine  Kac-Moody algebras.
We will use these materials throughout the
rest of this paper.

 Let $A=(a_{ij})_{i,j\in I}$, $\mu$ and $\fg$ be as in \S\,1.1.
We denote by $\ell$ the rank of $A$ and
 set  $I=\{1,\cdots,\ell\}$ (resp.\,$\{0,1,\cdots,\ell\}$) if $A$ is of finite (resp.\,affine) type.
 As a convention, we will  label the generalized Cartan matrix $A$  using Tables Fin and Aff 1-3 of \cite[Chap 4]{Kac-book}.
If $A$ has label $X_\ell$ (the finite case) or  $X_\wp^{(r)}$ (the affine case) with $\ell,\wp\ge 1$ and $r=1,2,3$,
we say that $\fg$ (or\,$A$) has type $X_\ell$  or  $X_\wp^{(r)}$.
As usual, we say that  $\fg$   is an untwisted (resp.\,twisted) affine Kac-Moody algebra if $A$ has type $X_\wp^{(r)}$ with $r=1$ (resp.\,$r> 1$).

Let $e_i^\pm, \al_i^\vee, i\in I$ be the Chevalley generators of $\fg$ as in \S\,1.1, and
$\fg=\fn_+\oplus \fh\oplus \fn_-$ the standard triangular decomposition of $\fg$, where
$\fh=\oplus_{i\in I}\C\al_i^\vee$ and  $\fn_+$, $\fn_-$ are the subalgebras of $\fg$ defined in \S\,1.1.
Let $\Delta$ be the root system (including $0$) of $\fg$ and $Q$
the root lattice of $\fg$. Then there is a natural $Q$-grading $\fg=\oplus_{\al\in Q}\fg_\al$ on $\fg$ whose support is $\Delta$.
Let $\Pi=\{\al_i, i\in I\}$ be the simple root system of $\fg$ such that $e_i^\pm\in \fg_{\pm\al_i}$ for $i\in I$,
let $\Delta_+$ be the set of positive roots and let $\Delta_-=-\Delta_+$.
Let $\Delta^{\times}$
 be the set of
real  roots in $\Delta$ and let $\Delta^0=\Delta\setminus \Delta^\times$.
Note that $\mu$ induces an automorphism on $Q$ such that $\mu(\al_i)=\al_{\mu(i)}$ for $i\in I$.

Let $D=\te{diag}\{\epsilon_i, i\in I\}$ be a diagonal matrix as in \S\,1.1.
According to \cite[Theorem 2.2, Exercise 2.5]{Kac-book}, there is a unique
invariant
 symmetric bilinear form $\<\cdot,\cdot\>$ on $\fg$ such that
\begin{align}\label{fgbiform}
\<\al_i^\vee,\al_j^\vee\>=\epsilon_j^{-1}a_{ij}\end{align}
 for $i,j\in I$.
 Recall the notion of  normalized invariant form on $\fg$ introduced in \cite[(6.2.1), (6.4.2)]{Kac-book}.
From now on, we will fix the choice of the matrix $D$ such that
if $\fg$ has type $X_{\ell}$ (resp.\,$X_{\wp}^{(r)}$), then $\<\cdot,\cdot\>$ is (resp.\,$r$ times of)
  the normalized invariant form on $\fg$.

\subsection{More on affine Kac-Moody algebras}
In this subsection, we assume that the algebra $\fg$ has affine type $X_\wp^{(r)}$.
Let $\dg$ be a finite dimensional simple Lie algebra of type $X_\wp$, $\dfh$ a Cartan subalgebra of it and
$\dot\nu$ a diagram automorphism of $\dg$ with order $r$.
For each $x\in \dg$ and $m\in \Z$, we set
\begin{align*}
x_{[m]}=r^{-1}\sum_{p\in \Z_r}\xi_r^{-mp}\dot{\nu}^p(x)\quad\te{and}\quad
\dg_{[m]}=\te{Span}_\C\{x_{[m]}\mid x\in \dg\}.
\end{align*}
Then it was shown in \cite[Chap.\,8]{Kac-book} that the affine Kac-Moody algebra $\fg$ can be realized as the Lie algebra
\[\mathrm{Aff}(\dg,\dot\nu)=(\sum_{m\in \Z}t_2^m\ot \dg_{[m]})\oplus \C\rk_2\]
 with the Lie bracket given by
\begin{align*}
[t_2^{m_1}\ot x+a_1\rk_2, t_2^{m_2}\ot y+a_2\rk_2]
=t_2^{m_1+m_2}\ot [x,y]+\<x,y\>\delta_{m_1+m_2,0}m_1\rk_2,
\end{align*}
where $m_1,m_2\in \Z$, $x\in \dg_{[m_1]},y\in \dg_{[m_2]}$, $a_1,a_2\in \C$
and $\<\cdot,\cdot\>$ denotes the normalized invariant form on $\dg$.
We remark that the invariant form $\<\cdot,\cdot\>$ on
   $\fg=\mathrm{Aff}(\dg,\dot{\nu})$  is given by (\cite[(8.3.8)]{Kac-book})
  \begin{align}\label{defaffbi}
  \<t_2^m\ot x, t_2^n\ot y\>=\<x,y\>\,\delta_{m+n,0},\quad \<\rk_2,\mathrm{Aff}(\dg,\dot{\nu})\>=0,
  \end{align}
for $m,n\in \Z$, $x\in \dg_{[m]}$ and $y\in \dg_{[n]}$.

  Let $\dot{\Delta}$ be the root system (containing $0$) of $\dg$ with respect to $\dfh$ and let
   $\dg=\oplus_{\dot\al\in \dot{\Delta}}\dg_{\dot\al}$ be the corresponding root spaces decomposition of $\dg$.
   Note that $\dot\nu$ induces an automorphism of the root lattice $\dot{Q}$ of $\dg$ such that $\dot{\nu}(\dot{\al}_i)=\dot{\al}_{\dot\nu(i)}$ for
$i\in \dot{I}$.
Then the root lattice
\[Q=\dQ_{[0]}\oplus \Z\delta_2,\]
 where $\dQ_{[0]}=\{\dot\al_{[0]}=r^{-1}\sum_{p\in \Z_r}\dot{\nu}^p(\dot\al)\mid \dot\al\in \dot{Q}\}$ and $\delta_2$ denotes the null root.

It was proved in \cite[\S\,2.2]{CJKT-uce} that there exists an automorphism $\dot\mu$ of $\dg$  and a homomorphism $\rho_\mu:\dQ\rightarrow \Z$
of abelian groups
such that
the action of $\mu$ on the real root spaces of $\fg=\mathrm{Aff}(\dg,\dot\nu)$ is as follows
\begin{equation}\begin{split}\label{actmu1}
t_2^m\ot x_{[m]}\mapsto t_2^{m+\rho_\mu(\dot\al)}\ot \dot\mu(x_{[m]}),
\end{split}\end{equation}
where $m\in \Z$, $\dot\al\in \dot{\Delta}\setminus\{0\}$ and $x\in \dg_{\dal}$.
We extend $\rho_\mu$ to a linear functional on $\dfh^*$ by $\C$-linearity and identify $\dfh$ with $\dfh^*$
by means of the normalized form on $\dg$.
Then $\rho_\mu$ can be viewed as a linear functional on $\dfh$, and the  action of $\mu$ on the imaginary root spaces of $\fg$
is as follows
\begin{align}\label{actmu2}
t_2^m\ot h_{[m]}\mapsto t_2^m\ot \dot\mu(h_{[m]})+\delta_{m,0}\, \rho_\mu(h)\, \rk_2,\quad \rk_2\mapsto \rk_2,
\end{align}
where $m\in \Z$ and $h\in \dfh$.

\subsection{Twisted toroidal Lie algebras}
We start with the definition of (twisted) multi-loop algebras.
 Let $\fk$ be a Lie algebra, and let
 $\sigma_1,\cdots,\sigma_s$ be pairwise commuting finite order automorphisms of it.
  Let
$\C[t_1^{\pm 1},\cdots, t_s^{\pm 1}]$
denote the algebra of Laurent polynomials in the commuting variables $t_1,\cdots,t_s$ over $\C$.
Then the multi-loop Lie algebra of $\fk$ related to $\sigma_1,\cdots,\sigma_s$
is by definition the following subalgebra
of $\C[t_1^{\pm 1},\cdots,t_s^{\pm1}]\ot \fk$:
\baa \mathcal L(\mathfrak k,\sigma_1,\cdots,\sigma_s)=\sum_{m_1,\cdots,m_s\in \Z}
t_1^{m_1}\cdots t_s^{m_s}\ot \mathfrak k_{(m_1,\cdots,m_s)},\eaa
where $\mathfrak k_{(m_1,\cdots,m_s)}=\{x\in \mathfrak k\mid \sigma_i(x)=\xi_{M_i}^{m_i}\ x,\ i=1,\cdots,s\}$ and $M_i$
is the order of $\sigma_i$.

Suppose now that $\fk$ is a finite dimensional simple Lie algebra, which is equipped with a
 normalized invariant form $\<\cdot,\cdot\>$. Let
 $\CK_{M_1,\cdots,M_s}$ be the complex vector space spanned by the symbols
\[t_1^{m_1}\cdots t_s^{m_s}\rk_1,\cdots,t_1^{m_1}\cdots t_s^{m_s}\rk_s,\]
 subject to the  relation
\[\sum_{i=1}^s m_i\, t_1^{m_1}\cdots t_s^{m_s}\rk_i=0,\]
where $m_i\in M_i\Z$ for all $i$.
We define the twisted toroidal  Lie algebra
\[\wh{\CL}(\fk,\sigma_1,\cdots,\sigma_s)=\CL(\fk,\sigma_1,\cdots,\sigma_s)\oplus \mathcal K_{M_1,\cdots,M_s},\]
with the Lie bracket given by
\begin{equation}\begin{split}\label{toroidalrel}
 &[t_1^{m_1}\cdots t_s^{m_s}\ot x, t_1^{n_1}\cdots t_s^{n_s}\ot y]\\
 =\quad &t_1^{m_1+n_1}\cdots t_s^{m_s+n_s}\ot [x,y]
+\<x,y\> (\sum_{i=1}^s m_it_1^{m_1+n_1}\cdots t_s^{m_s+n_s}\rk_i),\end{split}\end{equation}
where $x\in \fk_{(m_1,\cdots,m_s)}$, $y\in \fk_{(n_1,\cdots,n_s)}$, $m_1,\cdots, m_s,n_1,\cdots,n_s\in \Z$ and $\CK_{M_1,\cdots,M_s}$ is
the center space. It was proved in \cite{Sun-uce}
that  $\wh{\mathcal L}(\fk,\sigma_1,\cdots,\sigma_s)$ is central closed.

In this paper we will only use the algebra $\wh{\CL}(\fk,\sigma_1,\cdots,\sigma_s)$ for the special case that $s=1$ or $s=2$. Notice that, if $s=1$, $\CK_{M_1}=\C\rk_1$ is one dimensional.
And, if $s=2$, the elements
\begin{align}\label{centerbasis}
t_1^{m_1}t_2^{m_2}\rk_1,\quad \rk_1,\quad t_1^{n_1}\rk_2,\quad m_1, n_1\in M_1\Z,\ m_2\in M_2\Z^\times,\end{align}
form a basis of $\CK_{M_1,M_2}$.

\section{The Lie algebra $\wh\fg[\mu]$ and its Drinfeld type Serre relations}
In this section, we introduce the definition of the Lie algebra $\wh\fg[\mu]$ and establish a class of Drinfeld type Serre relations in $\wh\fg[\mu]$.

\subsection{The Lie algebra $\wh\fg[\mu]$}
From now on,  when $\fg$ is of affine type, we will often identify $\fg$ with $\mathrm{Aff}(\dg,\dot\nu)$ without further explanation.
We set
\begin{equation}\label{defwhfg}
\wh\fg=\begin{cases} \wh\CL(\fg,1),\ &\te{if $\fg$ is of finite type};\\
\wh\CL(\dg,1,\dot\nu),\ &\te{if $\fg$ is of affine type}.
\end{cases}
\end{equation}

As a convention, when $\fg$ is of affine type, we will also view $\CL(\fg)=\C[t_1,t_1^{-1}]\ot \fg$ as a subspace of
$\wh\fg$ in the following way
\begin{equation*}
t_1^{m_1}\ot x= t_1^{m_1}t_2^{m_2}\ot \dot x+at_1^{m_1}\rk_2,
\end{equation*}
for $x=t_2^{m_2}\ot \dot x+a\rk_2\in \fg,\ m_1\in \Z$.
For $m_1,m_2\in \Z$, we also set
\begin{equation*}
t_1^{m_1}t_2^{m_2}\rk_1'
=\begin{cases}0,\ &\te{if $\fg$ is of finite type};\\
0,\ &\te{if $\fg$ is of affine type and $m_2\notin r\Z^\times$};\\
\frac{1}{m_2} t_1^{m_1}t_2^{m_2}\rk_1,\ &\te{if $\fg$ is of affine type and $m_2\in r\Z^\times$}.\end{cases}
\end{equation*}
Then it follows from \eqref{centerbasis} that  the algebra $\wh\fg$ is spanned by the elements
\[t_1^{m_1}\ot x, \quad \rk_1,\quad \ t_1^{n_1}t_2^{n_2}\rk_1',\quad x\in \fg,\ m_1,n_1,n_2\in \Z.\]
Moreover, we have the following result.
\begin{lemt}\label{lem:commutator} Let $\al, \beta\in \Delta$, $x\in \fg_\al, y\in \fg_\beta$ and $m_1,n_1\in \Z$.
If $\al+\beta\in \Delta^\times\cup \{0\}$, then we have
\begin{align}\label{commutator1}
[t_1^{m_1}\ot x, t_1^{n_1}\ot y]=t_1^{m_1+n_1}\ot [x,y]+m_1\delta_{m_1,n_1}\<x,y\>\rk_1.\end{align}
If $\fg$ is of affine type, $x=t_2^{m_2}\ot \dot{x}$, $y=t_2^{n_2}\ot \dot{y}$ and $\al+\beta\in \Delta^0\setminus\{0\}$, then
\begin{align}
[t_1^{m_1}\ot x, t_1^{n_1}\ot y]=t_1^{m_1+n_1}\ot [x,y]
+\<\dot{x},\dot{y}\>(m_1n_2-m_2n_1)t_1^{m_1+n_1}t_2^{m_2+n_2}\rk_1'.
\end{align}
\end{lemt}
\begin{proof}A direct verification by using \eqref{toroidalrel}. Notice that, when $\fg$ is of affine type,
one needs to use the fact \eqref{defaffbi}.
\end{proof}

Now it is obvious that the map
\begin{align}
\psi:\wh\fg\rightarrow \CL(\fg),\quad t_1^m\ot x\mapsto t_1^{m_1}\ot x,\ \rk_1\mapsto 0,\ t_1^{n_1}t_2^{n_2}\rk_1'\mapsto 0
\end{align}
is the universal central extension of $\CL(\fg)$.
For $x\in \fg$ and $n\in \Z$, we introduce the  formal series  in $\wh\fg[[z,z^{-1}]]$ as follows:
\begin{align}\label{defxz}
x(z)=\sum_{m\in \Z}(t_1^m\ot x)z^{-m-1},\quad t_2^n\rk_1'(z)=\sum_{m\in \Z}(t_1^mt_2^n\rk_1')z^{-m},\quad \rk_1(z)=\rk_1.\end{align}
Using these formal series, one has the following reformulation of Lemma \ref{lem:commutator}.
\begin{lemt}\label{lem:relationwhfg}
Let $\al,\beta\in \Delta$, $x\in \fg_\al$ and $y\in \fg_\beta$. If $\al+\beta\in \Delta^\times\cup \{0\}$, then
\[[x(z),y(w)]=[x,y](w)z^{-1}\delta\(\frac{w}{z}\)+\<x, y\>\rk_1(w)\frac{\partial}{\partial w}z^{-1}\delta\(\frac{w}{z}\)\]
If $\fg$ is of affine type, $x=t_2^{m_2}\ot \dot{x}$, $y=t_2^{n_2}\ot \dot{y}$  and $\al+\beta\in \Delta^0\setminus\{0\}$, then
\begin{equation*}\begin{split}
[x(z), y(w)]=&\([x,y](w)+\<\dot{x},\dot{y}\>m_2\frac{\partial}{\partial w}t_2^{m_2+n_2}\rk_1'(w)\)
z^{-1}\delta\(\frac{w}{z}\)\\&+\<\dot{x},\dot{y}\>(m_2+n_2)t_2^{m_2+n_2}\rk_1'(w)\frac{\partial}{\partial w}z^{-1}\delta\(\frac{w}{z}\),
\end{split}\end{equation*}
where $\delta(z)=\sum_{m\in \Z}z^m$ is the usual $\delta$-function.
\end{lemt}

We define a $Q\times \Z$-grading
$\wh\fg=\oplus_{(\al,n)\in Q\times \Z}\ \wh\fg_{\al,n}$
on $\wh\fg$ by letting
\begin{align}\label{defpsi}
t_1^{m_1}\ot x\in \wh\fg_{\al,m_1},\quad \rk_1\in \wh\fg_{0,0},\quad t_1^{n_1}t_2^{n_2}\rk_1'\in \wh\fg_{n_2\delta_2,n_1},
\end{align} for $x\in \fg_\al$, $\al\in \Delta$ and $m_1, n_1, n_2\in \Z$.
This grading induces a natrual triangular decomposition
\begin{align}\label{decwhfg}
\wh\fg=\wh{\fn}_+\oplus \wh\fh\oplus \wh\fn_-\end{align}
of $\wh\fg$, where
$\wh\fh=\oplus_{m\in \Z}\,\wh\fg_{0,m}=\te{Span}_\C\{t_1^m\ot h,\ \rk_1\mid h\in  \fh, m\in \Z\}$ is a Heisenberg algebra
 and $\wh{\fn}_\pm=\oplus_{\al\in \Delta_\pm, m\in \Z}\,\wh\fg_{\al,m}.$
When $\fg$ is of finite type,  $\wh\fn_\pm=\C[t_1,t_1^{-1}]\ot \fn_\pm$ is the loop algebra of $\fn_\pm$.
However, when $\fg$ is of affine type,
\begin{align*}\wh\fn_\pm=\te{Span}_\C\{t_1^m\ot x,\ t_1^{m_1}t_2^{m_2}\rk_1'\mid x\in \fn_\pm, m,m_1\in \Z, \pm m_2\in \Z_+\}\end{align*}
is an (infinite dimensional) central extension of $\C[t_1,t_1^{-1}]\ot \fn_\pm$.

Observe that  the algebra $\wh\fg$ is generated by the following elements
\begin{align}\label{genhatg}
t_1^m\ot e_i^\pm,\quad t_1^m\ot \al_i^\vee,\quad \rk_1,\quad i\in I,\ m\in \Z.\end{align}
As was indicated in \eqref{defhatmu0}, we have the following result.
\begin{lemt}\label{lem:wh-mu-N-period}
The action
\begin{equation}\label{defhatmu}
t_1^m\ot e_i^\pm\mapsto \xi^{-m} t_1^m\ot e_{\mu(i)}^\pm,\quad t_1^m \al_i^\vee\mapsto \xi^{-m}t_1^m \al_{\mu(i)}^\vee,\quad
\rk_1\mapsto \rk_1
\end{equation}
for $i\in I,\ m\in \Z$, defines (uniquely) an automorphism $\wh\mu$ of $\wh\fg$.
 \end{lemt}
 \begin{proof} The lemma is obviously when $\fg$ is of finite type.
 For the affine case, this lemma was proved in  \cite[Lemma 3.2]{CJKT-uce}. For convenience of the readers,
 we describe the explicit action of $\wh\mu$ on $\wh\fg$ as follows:
 \begin{align}\label{defhatmu1}
 t_1^{m_1}\ot x &\mapsto \xi^{-m_1} t_1^{m_1}\ot \mu(x),\quad \rk_1\mapsto \rk_1,
 \end{align}
 where $m_1\in \Z$, $x\in \fg_\al$ and $\al\in \Delta^\times\cup\{0\}$, and  (the following actions exist only when $\fg$
 is of affine type)
 \begin{align}\label{defhatmu2}
 t_1^{m_1}\ot h &\mapsto  \xi^{-m_1}(t_1^{m_1}\ot \mu(h)
-m_1\rho_\mu(\dot{h})t_1^{m_1}t_2^{m_2}\rk_1'),\\
 t_1^{n_1}t_2^{n_2}\rk_1' &\mapsto  \xi^{-n_1}t_1^{n_1}t_2^{n_2}\rk_1',
 \end{align}
 where $m_1,n_1,n_2\in \Z$, $h=t_2^{m_2}\ot \dot{h}$, $m_2\in \Z^\times$ and $\dot{h}\in \dfh_{[m_2]}$.
 \end{proof}

We define $\wh\fg[\mu]$ to be the subalgebra of $\wh\fg$ fixed by  $\wh\mu$.
It was known (\cite{Kac-book,CJKT-uce}) that the restriction map (see \eqref{defpsi})
\begin{align*}
\psi|_{\wh\fg[\mu]}:\wh\fg[\mu]\rightarrow \CL(\fg,\mu)
\end{align*}
is the universal central extension of  $\CL(\fg,\mu)$.
We remark that the automorphism $\wh{\mu}$ preserves the  decomposition \eqref{decwhfg} of $\wh\fg$.
So we have
\[\wh\fg[\mu]=\wh\fn_+[\mu]\oplus \wh\fh[\mu]\oplus \wh\fn_-[\mu],\]
where $\wh\fn_\pm[\mu]=\wh\fg[\mu]\cap \wh\fn_\pm$ and $\wh\fh[\mu]=\wh\fg[\mu]\cap \wh\fh$.
Using \eqref{genhatg} and \eqref{defhatmu}, one knows that the Lie algebra $\wh\fg[\mu]$ is generated by the following elements
\begin{align} \label{genwhfgmu} t_1^m\ot e_{i(m)}^\pm,\quad t_1^m\ot \al_{i(m)}^\vee,\quad \rk_1,\quad i\in I,\ m\in \Z,
\end{align}
where the notation $x_{(m)}$ was given in \eqref{fm}.
Moreover, the subalgebras $\wh\fn_\pm[\mu]$ (resp. $\wh\fh[\mu]$) of $\wh\fg$ are generated by the elements
$t_1^m\ot e_{i(m)}^\pm$ (resp. $t_1^m\ot \al_{i(m)}^\vee, \rk_1$) for $i\in I$ and $m\in \Z$.
The following result can be checked directly
by using  \eqref{commutator1}.

\begin{lemt} \label{lem:orre}
      Let $i,j\in I$ and $m,n\in \Z$. Then
\begin{align*}
&t_1^m\ot \al_{\mu(i)(m)}^\vee=\xi^{m}t_1^m\ot \al_{i(m)}^\vee,\quad t_1^m\ot e_{\mu(i)(m)}^\pm=\xi^m t_1^m\ot e_{i(m)}^\pm,\\
&[t_1^m\ot \al_{i(m)}^\vee, \rk_1]=0=[t_1^m\ot e_{i(m)}^\pm, \rk_1],\\
&[t_1^m\ot \al_{i(m)}^\vee, t_1^n\ot \al_{j(n)}^\vee]=\sum_{k\in \Z_N}\frac{mN}{\epsilon_j}a_{i\mu^k(j)}\xi^{km}\delta_{m+n,0}\rk_1,\\
&[t_1^{m}\ot \al_{i(m)}^\vee, t_1^{n}\ot e_{j(n)}^\pm]=\pm\sum_{k\in \Z_N} a_{i\mu^k(j)}\xi^{km} t_1^{m+n}\ot e_{j(m+n)}^\pm,\\
&[t_1^{m}\ot e_{i(m)}^+, t_1^n\ot e_{j(n)}^-]=\sum_{k\in \Z_N}\delta_{i,\mu^k(j)}\xi^{km}(t_1^{m+n}\ot \al_{j(m+n)}^\vee
+\frac{mN}{\epsilon_j}\delta_{m+n,0}\rk_1).
\end{align*}
\end{lemt}

As in \eqref{defxz}, for $x\in \fg$ and $n\in \Z$, we  introduce the formal series in $\wh\fg[\mu][[z,z^{-1}]]$ as follows:
\begin{align*}
x(z)=\sum_{m\in\Z} \(\sum_{p\in \Z_N}\wh\mu^p(t_1^m\ot x)\) z^{-m-1},\
t_2^{n}\rk_1'(z)=\sum_{m\in N\Z}t_1^mt_2^n\rk_1'z^{-m},\ \rk_1(z)=\rk_1.
\end{align*}
Note that, if $x\in \fg_\al$ with $\al\in \Delta^\times\cup\{0\}$, then it follows from \eqref{defhatmu1} that
\[x(z)=\sum_{m\in \Z}t_1^m\ot x_{(m)}\,z^{-m-1}.\]
 This is not true when  $\fg$ is of affine type and $x\in \fg_\al$ with $\al\in \Delta^0\setminus\{0\}$.
However,  the following still holds true (see \eqref{defhatmu2})
\begin{align*}
[x(z),y(w)]=[\sum_{m\in \Z}t_1^m\ot x_{(m)}\,z^{-m-1}, \sum_{n\in \Z}t_1^n\ot y_{(n)}\,w^{-n-1}]
\end{align*}
for all $x,y\in \fg$.
Using this observation and Lemma \ref{lem:relationwhfg}, one can easily verify the following result.
\begin{lemt}\label{comminghmu} For $\al,\beta\in \Delta$, $x\in \fg_\al$ and $y\in \fg_\beta$, one has that
\begin{align*}
&[x(z),y(w)]\\
=&\sum_{k\in \Z_N}[\mu^k(x),y](w)z^{-1}\delta\(\xi^{-k}\frac{w}{z}\)\\
&+\sum_{k\in J_0} \<\mu^k(x), y\>\rk_1(w)\frac{\partial}{\partial w}z^{-1}\delta\(\xi^{-k}\frac{w}{z}\)\\
&+\sum_{k\in J_1}\<\dot{x}_k,\dot{y}\>\(m_k\(\frac{\partial}{\partial w}t_2^{m_k+n}\rk_1'(w)\)
z^{-1}\delta\Big(\xi^{-k}\frac{w}{z}\)\\
&+(m_k+n)t_2^{m_k+n}\rk_1'(w)
\frac{\partial}{\partial w}z^{-1}\delta\(\xi^{-k}\frac{w}{z}\)\Big),
\end{align*}
Here, $J_0=\{k\in \Z_N\mid \mu^k(\al)+\beta=0\}$ and $J_1=\{k\in \Z_N\mid \mu^k(\al)+\beta\in \Delta^0\setminus\{0\}\}$.
The set $J_1$ exists only when $\fg$ is of affine type, and in this case
the notations $\dot{x}_k,\dot{y}\in \dg$, $m_k,n\in \Z$, $k\in J_1$ are defined by the rule
\begin{align*}
\mu^k(x)=t_2^{m_k}\ot\dot{x}_k\quad \te{and}\quad y=t_2^n\ot \dot{y}.
\end{align*}
\end{lemt}

Finally, recall the Lie algebras $\mathcal D_{\bm{P}}(\fg,\mu)$, $\mathcal D_{\bm{P}}(\fn_+,\mu)$ and
 $\mathcal D_{\bm{P}}(\fn_-,\mu)$ defined in Definition \ref{defDPgmu}.
 Then we have the following result.

\begin{lemt}\label{thetahom} The assignment \eqref{thetafg} (resp.\,\eqref{thetafn})
 determines a surjective Lie homomorphism, say $\theta_{\fg,\bm{P}}$ (resp.\,$\theta_{\fn_+,\bm{P}}$ or $\theta_{\fn_-,\bm{P}}$),
from $\mathcal D_{\bm{P}}(\fg,\mu)$ (resp.\,$\mathcal D_{\bm{P}}(\fn_+,\mu)$ or  $\mathcal D_{\bm{P}}(\fn_-,\mu)$)
 to $\wh\fg[\mu]$ (resp.\,$\wh\fn_+[\mu]$ or $\wh\fn_-[\mu]$).
\end{lemt}
\begin{proof} It suffices to show that the elements in \eqref{genwhfgmu} satisfy the defining relations
of $\mathcal D_{\bm{P}}(\fg,\mu)$.
The  relations $(H),(XX),(HX\pm)$ are implied by Lemma \ref{lem:orre},
the relation $(X\pm)$ is implied by Lemma \ref{comminghmu}
and  the relation $(DS\pm)_{\bm{P}}$ is implied by the condition
$(\bm{P}1)$. Now we check the relation $(AS\pm)$, and so $\fg$ is of type $A_1^{(1)}$.
Note that in this case $\al_{\mu^{k_1}(i)}+\al_{\mu^{k_2}(i)}+\al_j$ is not a null root for any
$i_1,i_2\in \Z_N$ and $(i,j)\in \mathbb I$. This together with Lemma \ref{comminghmu} implies that
\begin{align*}
&[e_i^\pm(z_1),[e_i^\pm(z_2),e^\pm_j(w)]]\\
=&\sum_{k_1,k_2\in \Z_N}[e_{\mu^{k_1}(i)}^\pm,[e_{\mu^{k_2}(i)}^\pm,e_j^\pm]](w)
z_1^{-1}\delta\(\xi^{-k_1}\frac{w}{z_1}\)z_2^{-1}\delta\(\xi^{-k_2}\frac{w}{z_2}\).\end{align*}
Therefore, we have that
\[(z_1^N-z_2^N)\cdot [e_i^\pm(z_1),[e_i^\pm(z_2),e^\pm_j(w)]]=0,\]
as required.
\end{proof}

\begin{remt} Let $i,j\in I$. Then $\al_i+\al_j\in \Delta^0$ if and only if $\fg$ is of type $A_1^{(1)}$ and $i\ne j$.
Therefore, when $\fg$ is of type $A_1^{(1)}$, the defining relation $(X\pm)$ of $\mathcal D_{\bm{P}}(\fg,\mu)$
 is different from that of other types.
\end{remt}

\subsection{More on diagram automorphisms}
Before deducing the Drinfeld type Serre relations in $\wh\fg[\mu]$,
 in this subsection we collect some elementary properties  on the  automorphism $\mu$ for later use.

Let $\<\mu\>$ be the subgroup of the permutation group on $I$ generated by $\mu$.
For $i\in I$, let $\mathcal{O}(i)$ be the orbit of $I$  containing $i$ under the action of  $\<\mu\>$.
The following result
is about the edges joining the vertices in the same  orbit of $I$.
\begin{lemt}\label{lem:defsi}
 For $i\in I$,  exactly one of the following holds
\begin{enumerate}
\item[(a)] the elements $\alpha_p, p\in \mathcal{O}(i)$ are pairwise orthogonal;
\item[(b)] $\mathcal{O}(i)=\{i,\mu(i)\}$ and
$a_{i\mu(i)}=-1=a_{\mu(i)i}$;
\item[(c)] $A$ has type $A_{\wp}^{(1)}$ and $\mathcal{O}(i)=I$.
\end{enumerate}
\end{lemt}
\begin{proof}
The lemma is proved by checking the claim for each possible $A$ and
each  diagram automorphism $\mu$ on $A$.
For a list of  automorphisms on affine generalized Cartan matrices, see \cite[Tables 2, 3]{ABP} for example.
\end{proof}

For $i\in I$, set
\begin{equation}\label{defdi}
 s_i=\begin{cases}1,\ \text{if (a) holds in Lemma \ref{lem:defsi}};\\
2,\ \text{if (b) holds in Lemma \ref{lem:defsi}};\\
3,\ \text{if (c) holds in Lemma \ref{lem:defsi}}.\end{cases}\end{equation}
The automorphism $\mu$ is said to be transitive if $\mu$ acts transitively on the set $I$.
Observe that, if $A$ is of affine type, then $\mu$ is non-transitive if and only if $s_i\le 2$ for all $i\in I$.

For  $i,j\in I$, we set
\begin{align}\label{defgij}\Gamma_{ij}^-=\{k\in \Z_N\mid a_{i\mu^k(j)}< 0\},\end{align}
and introduce the numbers
\begin{align}\label{defdij}N_i=\mathrm{Card}\,\mathcal O_i,\quad d_i=N/N_i,\quad N_{ij}=\gcd(N_i,N_j),\quad d_{ij}=\mathrm{Card}\, \Gamma_{ij}^-.\end{align}
The following result is about the edges joining the vertices in two different orbits.
\begin{lemt}\label{lem:linking2}
Let  $i,j\in I$ and $k\in \Z_N$. Assume that $i\notin\mathcal O(j)$ and $a_{ij}<0$.
Then exactly one of the following holds
\begin{enumerate}
\item[(a)] $a_{i\mu^k(j)}=0$;
\item[(b)] $a_{i\mu^k(j)}=a_{ij}$ and $N_{ij}$ divides $k$.
\end{enumerate}
In particular, in this case $\Gamma_{ij}^-$ is a subgroup of $\Z_N$ with order $d_{ij}=N/N_{ij}$.
\end{lemt}
\begin{proof}
Let $i,j$ be as in lemma. Note that in this case  $\{s_i,s_j\}=\{1\}$ or $\{1,2\}$, and so we may (and do) assume that $s_j=1$.
In what follows we list all the possible Dynkin diagrams $S_{ij}$ of $\mathcal O(i)\uplus \mathcal O(j)$ and the numbers $N_{ij}$:
\begin{enumerate}
\item[(i)] if $s_i=1$, then $S_{ij}$ is one of the types
 $A_2, A_3, D_4, D_{4}^{(1)}, A_2\times A_2,  B_2\times B_2$ and $A_2\times A_2\times A_2$ and in each case
$N_{ij}=1,1,1,1,2,2$ and $3$, respectively;
\item[(ii)] if $s_i=2$, then $S_{ij}$ is one of the types $A_4,  C_3^{(1)}$, $D_4^{(2)}, D_5^{(1)}$ and $A_2^{(1)}$ and
in each case $N_{ij}=2,2,2,2$ and  $1$, respectively.
\end{enumerate}
Using this and a case by case argument, one can get the desired result.
\end{proof}

Recall from \S\,1.1 that $\mathbb I=\{(i,j)\in I\times I\mid a_{ij}<0\}$.
For $(i,j)\in \mathbb I$ and $\bm{k}=(k_1,\cdots,k_{1-a_{ij}})\in (\Z_N)^{1-a_{ij}}$, we introduce the notation
\[\al_{ij}(\bm{k})=\al_{\mu^{k_1}(i)}+\cdots+\al_{\mu^{k_{1-a_{ij}}}(i)}+\al_j\in Q.\]
The rest of this subsection is devoted to give the sufficient and necessary condition for
$\al_{ij}(\bm{k})$ belongs to $\Delta$ (resp.\,$\Delta^\times$;
resp.\,$\Delta^0$). Firstly, one has that

\begin{lemt}\label{alijk} Let $(i,j)\in \mathbb I$ and $\bm{k}=(k_1,\cdots,k_{1-a_{ij}})\in (\Z_N)^{1-a_{ij}}$. Then
\begin{enumerate}
\item[(a)] if $s_i=1$, then  $\al_{ij}(\bm{k})\in \Delta$
 if and only if   $-k_1,\cdots,-k_{1-a_{ij}}\in \Gamma_{ij}^-$ and at least two of
$\mu^{k_1}(i),\cdots,\mu^{k_{1-a_{ij}}}(i)$ are distinct;
\item[(b)] if $s_i=2$ and $i\in \mathcal O(j)$, then $\al_{ij}(\bm{k})\not\in \Delta$ for all $\bm{k}\in (\Z_N)^{1-a_{ij}}$;
\item[(c)] if $s_i=2$, $i\notin \mathcal O(j)$ and $N_{ij}=2$, then $\al_{ij}(\bm{k})\in \Delta$ if and only if
all except one of $k_1,\cdots,k_{1-a_{ij}}$ contained in $\Gamma_{ij}^-$;
\item[(d)] if $s_i=2$, $i\notin \mathcal O(j)$ and $N_{ij}=1$, then $\al_{ij}(\bm{k})\in \Delta$ if and only if
$k_1\ne k_2$;
\item[(e)] if $s_i=3$ and $\wp=1$, then $\al_{ij}(\bm{k})\in \Delta$ if and only if there is exactly one, say $k$,  of $k_1,k_2,k_3$ such that
$\mu^{k}(i)=j$;
\item[(f)] if $s_i=3$ and $\wp=2$, then $\al_{ij}(\bm{k})\in \Delta$ if and only if $\mu^{k_1}(i),\mu^{k_2}(i),j$ are pairwise distinct;
\item[(g)] if $s_i=3$ and $\wp\ge 3$, then $\al_{ij}(\bm{k})\in \Delta$ if and only if the Dynkin diagram of
 $\{\mu^{k_1}(i),\mu^{k_2}(i),j\}$ is of type $A_3$.
 \end{enumerate}
\end{lemt}
\begin{proof} We first prove the assertion (a) and so assume now that $s_i=1$.
Let $\al=(\sum_{p\in \mathcal O(i)}m_p \al_p)+\al_j$ for some $m_p\in \N$.
By using Lemma \ref{lem:defsi} (a) and an induction argument on the number $N_i$, one can easily check that
$\al\in\Delta$ if and only if $m_p\le a_{pj}$ for all $p\in \mathcal O(i)$.
It is easy to see that this implies the assertion (a) and we omit the
details.
The assertion (b) is implied by Lemma \ref{lem:defsi} (b), and the assertions (c),(d) can be proved by
checking all the possible Dynkin diagrams $S_{ij}$ of $\mathcal O(i)\cup \mathcal O(j)$ given in the proof of
Lemma \ref{lem:linking2}.

For the assertion (e), one only needs to notice that if $\fg$ is of type $A_1^{(1)}$, then $m_0\al_0+m_1\al_1\in \Delta$ for some $m_0,m_1\in \N$
with $m_0+m_1=4$ if and only if $m_0=m_1=2$.
The assertion (f) follows from the fact that: if $\fg$ is of type $A_2^{(1)}$, then
$m_0\al_0+m_1\al_1+m_2\al_2\in \Delta$ for some $m_0,m_1,m_2\in \N$ with $m_0+m_1+m_2=3$ if and only
if $m_0=m_1=m_2=1$. Finally, the assertion (g) is obvious.
\end{proof}

As an immediate by-product of Lemma \ref{alijk}, we have  that
\begin{cort}\label{ksnekt} Let $(i,j)\in \mathbb I$ and $\bm{k}=(k_1,\cdots,k_{1-a_{ij}})\in (\Z_N)^{{1-a_{ij}}}$. If $\al_{ij}(\bm{k})\in \Delta$,
then there exist $s,t=1,\cdots,1-a_{ij}$ such that $\mu^{k_s}(i)\ne \mu^{k_t}(i)$.
\end{cort}

Now, for  $(i,j)\in \mathbb I$, we define three subsets of $(\Z_N)^{1-a_{ij}}$ as follows
\begin{align*}
\Upsilon_{ij}
& =\{\bm{k}\in (\Z_N)^{1-a_{ij}}\mid \al_{ij}(\bm{k})\in \Delta\};\\
\Upsilon_{ij}^\times
& =\{\bm{k}\in (\Z_N)^{1-a_{ij}}\mid \al_{ij}(\bm{k})\in \Delta^\times\};\\
\Upsilon_{ij}^0
& =\{\bm{k}\in (\Z_N)^{1-a_{ij}}\mid \al_{ij}(\bm{k})\in \Delta^0\}.
\end{align*}
The following result, together with Lemma \ref{alijk}, determines the three sets given above.
\begin{lemt} Let $(i,j)\in \mathbb I$. Then the following results hold true
 \begin{enumerate}
\item[(a)] if $A$ is not one of types $A_{1}^{(1)}$ and $A_2^{(1)}$, then $\Upsilon_{ij}=\Upsilon_{ij}^\times$;
\item[(b)] if  $A$ is of type $A_{1}^{(1)}$ or $A_2^{(1)}$, then $\Upsilon_{ij}=\Upsilon_{ij}^0$.
\end{enumerate}
\end{lemt}
\begin{proof} Assume that $\al_{ij}(\bm{k})=(\sum_{p\in \mathcal O(i)}m_p\al_p)+\al_j\in \Delta^0$ and so $\fg$ is of affine type.
Let $a_s, s\in I$ be the labels in the diagrams of \cite[Chap.4, Table Aff 1-3]{Kac-book}.
Recall that $a_s, s\in I$ are the (uniquely) positive integers such that $\delta_2=\sum_{s\in I}a_s\al_s$.
In view of this, if $j\notin\mathcal O(i)$, then $\al_{ij}(\bm{k})\in \Delta^0$ if and only if
\[I=\mathcal O(i)\cup \{j\},\ \mathcal O(j)=\{j\},\ a_j=1\ \te{and}\ m_p=a_p\ \te{for all}\ p\in \mathcal O(i).\]
Using the Dynkin diagrams $S_{ij}$ given in the  proof of Lemma \ref{lem:linking2}, one can easily check that $\al_{ij}(\rk)\in \Delta^0$ if and only
if $\fg$ is of type $A_{2}^{(1)}$ with $N=2$, $N_i=2$ and $N_j=1$.
For the case  $j\in \mathcal O(i)$, one has that $I=\mathcal O(i)$, $\fg$ has type $A_{N-1}^{(1)}$ and $a_s=1$ for all $s\in I$.
Then  the assertion immediately follows from Lemma \ref{alijk} (e),(f) and (g).
\end{proof}

\subsection{Drinfeld type Serre relations in $\wh\fg[\mu]$}
In this subsection we construct a class of  Drinfeld type Serre relations in $\wh\fg[\mu]$.

For  $(i,j)\in \mathbb I$, we define the Drinfeld polynomial
 \begin{align}\label{eq:def-p-ij}
 p_{ij}(z,w)=\prod_{k\in \Omega_{ij}^\times}(z-\xi^k w)\cdot\prod_{k\in \Omega_{ij}^0}(z-\xi^k w)^2,
 \end{align}
where  $\Omega_{ij}^\times$ and $\Omega_{ij}^0$ are the subsets of $\Z_N$ defined as follows:
for any $k\in \Z_N$, $k\in \Omega_{ij}^\times$ (resp.\,$\Omega_{ij}^0$) if and only if there is a tuple $\bm{k}\in
\Upsilon_{ij}^\times$ (resp.\,$\Upsilon_{ij}^0$)
such that $k=k_s-k_t$ for some $1\le s,t\le m_{ij}$ with $\mu^{k_s}(i)\ne \mu^{k_t}(i)$.

Using Lemma \ref{lem:defsi}, Lemma \ref{lem:linking2} and Lemma \ref{alijk}, we have the following explicit description of
the polynomials  $p_{ij}(z,w)$.
\begin{lemt} \label{lem:expij} Let $(i,j)\in \mathbb I$. If $j\in \mathcal O(i)$, then
\begin{equation*} p_{ij}(z,w)= \begin{cases}
1,\ &\te{if}\ s_i\le 2;\\
\(\frac{z^N-w^N}{z-w}\)^2, \ &\te{if}\ N=2,3\ \te{and}\ s_i=3;\\
\frac{z^N-w^N}{z-w}, \ &\te{if}\ N=4,5\ \te{and}\ s_i=3;\\
\prod_{k\in \Gamma_{ii}^-}(z-\xi^k w)(z-\xi^{2k} w), \ &\te{if}\ N\ge 6\ \te{and}\ s_i=3.
\end{cases}
\end{equation*}
And, if $j\notin\mathcal O(i)$, then
\begin{equation*}\begin{split} p_{ij}(z,w)&= \frac{z^{s_id_i}-w^{s_id_i}}{z^{d_i}-w^{d_i}}
\cdot\frac{z^{d_{ij}}-w^{d_{ij}}}{z^{d_i}-w^{d_i}}\\
&=\begin{cases}
\frac{z^{N/N_{ij}}-w^{N/N_{ij}}}{z^{N/N_i}-w^{N/N_i}},\ &\te{if}\ s_i=1;\\
z^{N/N_i}+w^{N/N_i},\ &\te{if}\ s_i=2\ \te{and}\ N_{ij}=2;\\
(z+w)^2, \ &\te{if}\ s_i=2\ \te{and}\ N_{ij}=1.
\end{cases}\end{split}
\end{equation*}
\end{lemt}

When $\fg$ is of finite type and $(i,j)\in \mathbb I$ with $i\notin \mathcal O(j)$,
it follows from Lemma \ref{lem:expij} that
the polynomial $p_{ij}(z,w)$ defined in \eqref{eq:def-p-ij} coincides with that appeared in the relation $(DS\pm)$ (see
Theorem \ref{intro1}). For this reason, $p_{ij}(z,w)$ is called the the Drinfeld polynomial.

Now we establish a class of Drinfeld type Serre relations in $\wh\fg[\mu]$ as follows.

\begin{prpt}\label{D-S-R} For $(i,j)\in \mathbb I$, one has that
\begin{align}
\prod_{1\le s<t\le 1-a_{ij}}p_{ij}(z_s,z_t)\cdot [e_i^\pm(z_1),\cdots,[e_i^\pm(z_{1-a_{ij}}),e_j^\pm(w)]]=0.
\end{align}
\end{prpt}
\begin{proof}
The proposition is implied by the definition of $p_{ij}(z,w)$, Corollary \ref{ksnekt} and the lemma below.
\end{proof}

\begin{lemt} Let $(i,j)\in \mathbb I$.
If $A$ is not of type $A_{1}^{(1)}$ or $A_2^{(1)}$, then
  \begin{align*}
 &[e_i^\pm(z_1),\cdots,[e_i^\pm(z_{1-a_{ij}}),e_j^\pm(w)]]\\
 =\,&\sum_{\bm{k}=(k_1,\cdots,k_{1-a_{ij}})\in \Upsilon_{ij}^\times} [e^\pm_{\mu^{k_1}(i)},\cdots, [e_{\mu^{k_{m_{ij}}}(i)}^\pm, e_j^\pm]](w)
 \prod_{1\le i\le s}
z_i^{-1}\delta\(\xi^{-k_i}\frac{w}{z_i}\).
 \end{align*}
If $A$ is of type $A_{1}^{(1)}$ or $A_2^{(1)}$, then
 \begin{align*}
 &\ [e_i^\pm(z_1),\cdots,[e_i^\pm(z_{1-a_{ij}}),e_j^\pm(w)]]\\
 =\,&\sum_{\bm{k}=(k_1,\cdots,k_{1-a_{ij}})\in \Upsilon_{ij}^0}\( [e^\pm_{\mu^{k_1}(i)},\cdots, [e_{\mu^{k_{1-a_{ij}}}(i)}^\pm, e_j^\pm]](w)\pm
 \<\dot{x}^\pm_{k_1},\dot{y}^\pm_{\bm{k}'}\>
 m_{k_1}\frac{\partial}{\partial w}\rk_1'(w)\) \\
 &\cdot \prod_{1\le i\le 1-a_{ij}}
z_i^{-1}\delta\(\xi^{-k_i}\frac{w}{z_i}\)
+\<\dot{x}^\pm_{k_1},\dot{y}^\pm_{\bm{k}'}\>\rk_1'(w)
 \(\xi^{-k_1}\frac{\partial}{\partial w}z_1^{-1}\delta\(\xi^{-k_1}\frac{w}{z_1}\)\)\\
 &\cdot \prod_{2\le i\le 1-a_{ij}}
z_i^{-1}\delta\(\xi^{-k_i}\frac{w}{z_i}\),
 \end{align*}
 where the elements  $m_{k_1},n_{\bm{k}'}\in\Z$, $\dot{x}_{k_1}^\pm, \dot{y}^\pm_{\bm{k}'}\in \dg$
 are determined by the following rule
 \[e^\pm_{\mu^{k_1}(i)}=t_2^{\pm m_{k_1}}\ot \dot{x}^\pm_{k_1},\quad [e^\pm_{\mu^{k_2}(i)},\cdots,
 [e_{\mu^{k_{1-a_{ij}}}(i)}^\pm, e_j^\pm]]=t_2^{\pm n_{\bm{k}'}}\ot \dot{y}^\pm_{\bm{k}'}.\]

\end{lemt}
\begin{proof} It is directly verified by using Lemma \ref{comminghmu}.
\end{proof}

For $(i,j)\in \mathbb I$ and $\sigma\in S_{1-a_{ij}}$, we set
\begin{equation}
p_{ij,\sigma}(z_1,\cdots,z_{1-a_{ij}})=\begin{cases}\prod_{1\le s<t\le 1-a_{ij}}p_{ij}(z_s,z_t),\ &\te{if}\ \sigma=1;\\
0,\ &\te{if}\ \sigma\ne 1.
\end{cases}
\end{equation}
Then it follows from Proposition \ref{D-S-R} that the family
\begin{align}
\bm{p}=\{P_{ij,\sigma}(z_1,\cdots,z_{1-a_{ij}},w)=p_{ij,\sigma}(z_1,\cdots,z_{1-a_{ij}})\mid (i,j)\in \mathbb I,\ \sigma\in S_{1-a_{ij}}\}
\end{align}
of homogenous polynomials satisfies the conditions $(\bm{P}1)$ and $(\bm{P}2)$.
Notice that, when $\mu=1$, the Serre relations $(DS\pm)_{\bm{p}}$ are the usual Serre relations
\begin{align}\label{depun}
[x_i^\pm(z_1),\cdots,[x_i^\pm(z_{1-a_{ij}}),x_j^\pm(w)]]=0,\quad (i,j)\in \mathbb I.
\end{align}

More generally, for each $(i,j)\in \mathbb I$, let $f_{ij}(z_1,\cdots,z_{1-a_{ij}},w)$ be  an arbitrarily given homogenous polynomial  such
that $f(w,\cdots,w,w)\ne 0$.
Then, according to Theorem \ref{intromain1} and Proposition \ref{D-S-R}, one knows that $\mathcal D_{\bm{f}}(\fg,\mu)$ is a Drinfeld type presentation
of $\wh\fg[\mu]$, where
\begin{align*}
\bm{f}=\{f_{ij,\sigma}(z_1,\cdots,z_{1-a_{ij}},w)\mid (i,j)\in \mathbb I,\ \sigma\in S_{1-a_{ij}}\},
\end{align*}
with
\begin{equation*}
f_{ij,\sigma}(z_1,\cdots,z_{1-a_{ij}},w)=f_{ij}(z_1,\cdots,z_{1-a_{ij}},w)\cdot p_{ij,1}(z_1,\cdots,z_{1-a_{ij}}),
\end{equation*}
if $\sigma=1$ and $f_{ij,\sigma}(z_1,\cdots,z_{1-a_{ij}},w)=0$ if $\sigma\ne 1$.
Thus one can obtain in this way a large class of Drinfeld type presentations for $\wh\fg[\mu]$.


\section{Classical limit of twisted quantum affinization algebras}
In \cite{CJKT-twisted-quantum-aff-vr} we introduced a class of twisted  quantum affinization algebras  by constructing their vertex representations.
As an application of Theorem \ref{intromain1}, in this section we determine the classical limit of these quantum algebras.
Throughout this section, we assume that $A$ is simply-laced and that $\mu$ is not a transitive automorphism of $A_{\wp}^{(1)}$, $\wp\ge 2$.

We first recall the twisted quantum affinization algebra introduced in \cite{CJKT-twisted-quantum-aff-vr}.
Let $\hbar$ be an indeterminate over $\C$, and $\C[[\hbar]]$ the power series ring of $\hbar$.
Set $q=e^\hbar$, and
let
\begin{align*}
  [n]_q=\frac{q^n-q^{-n}}{q-q^{-1}}=
  \(\sum_{m\ge0}\frac{2n^{2m+1}}{(2m+1)!}\hbar^{2m}\)
  \(\sum_{m\ge0}\frac{2}{(2m+1)!}\hbar^{2m}\)^{-1}
  \in\C[[\hbar]]
\end{align*}
for $n\in \Z$.

For  $i,j\in I$,  we introduce the polynomials
\begin{align*}
&F^\pm_{ij}(z,w)=\prod_{k\in\Z_N;a_{i\mu^k(j)}\ne 0}\left( z-\xi^{k}q^{\pm a_{i\mu^k(j)}}w \right),\\
&G^\pm_{ij}(z,w)=\prod_{k\in\Z_N;a_{i\mu^k(j)}\ne 0}\left( q^{\pm a_{i\mu^k(j)}}z-\xi^{k}w \right).
\end{align*}
In the case that $a_{ij}<0$ and $i\notin\mathcal O(j)$, we also introduce the polynomials
\begin{align*}
p_{ij}^\pm(z,w)=
    \left(
        z^{d_{i}}+q^{\mp d_{i}}w^{d_{i}}
    \right)^{s_i-1}
    \frac{
        q^{\pm 2d_{ij}}z^{d_{ij}}-w^{d_{ij}}
    }{
        q^{\pm 2d_i}z^{d_{i}}-w^{d_i}
    }.
\end{align*}
Furthermore, for $i\in I$ with $s_i=2$, we set
\begin{align*}
p_i^\pm(z_1,z_2,z_3)=q^{\mp d_{i}}z_{1}^{d_{i}}
        -(q^{\pm d_{i}}+1)
            z_{2}^{d_{i}}
        +q^{\pm 2d_{i}}
            z_{3}^{d_{i}}.
            \end{align*}



The twisted quantum affinization algebra $\U_\hbar(\wh\fg_\mu)$ defined in \cite{CJKT-twisted-quantum-aff-vr} is a $\C[[\hbar]]$-algebra
topologically generated by the elements
\begin{eqnarray*}\label{eq:tqagenerators}
h_{i,m},\  x^\pm_{i,m},\ c,\quad i\in I,\, m\in\Z,
\end{eqnarray*}
and subject to the relations
\begin{align*}
&\te{(Q0)  }&&x^\pm_{\mu(i),m}=\xi^{m} x^\pm_{i,m},
    \quad h_{\mu(i),m}=\xi^m h_{i,m},\\
&\te{(Q1)  }&&c \te{ is central},\\
&\te{(Q2)  }&&[h_{i,0}, h_{j,m}]=0,\\
&\te{(Q3)  }&&[h_{i,0}, x^\pm_{j,m}]=\pm\sum_{k\in \Z_N}
    a_{i\mu^k(j)}x_{j,m}^\pm,\\
&\te{(Q4) }&& [h_{i,m},h_{j,m'}]=\delta_{m+m',0}\frac{1}{m}
    \sum_{k\in \Z_N} \xi^{mk}[ma_{i\mu^k(j)}]_q\frac{q^{mc}-q^{-mc}}{q-q^{-1}},\,
    \te{if }m\ne 0,\\
&\te{(Q5) }&&[h_{i,m},x_{j,n}^\pm]
=\pm \frac{1}{m}\sum_{k\in \Z_N} \xi^{mk}[ma_{i\mu^k(j)}]_qq^{\mp\frac 1 2 |m|c} x_{j,m+n}^\pm,\,\te{if }m\ne 0,
    \\
&\te{(Q6)  }&&[x_i^+(z),x_j^-(w)]
=\frac{1}{q-q^{-1}}
    \sum_{k\in \Z_N}\delta_{i,\mu^k(j)}\\
 &&&  \quad\times \Bigg(
        \phi_i^+(q^{-\frac 1 2 c}z)\delta\left(
            \frac{q^c\xi^kw}{z}
        \right)
        -
        \phi_i^-(q^{\frac 1 2 c}z)\delta\left(
            \frac{q^{-c}\xi^kw}{z}
        \right)
    \Bigg),\\
&\te{(Q7)  }&&F^\pm_{ij}(z,w)x^\pm_i(z)x^\pm_j(w)=
    G^\pm_{ ij}(z,w)x^\pm_j(w)x^\pm_i(z),\\
&\te{(Q8)  }&&\sum_{\sigma\in S_{2}}\Big\{
    p_{ij}^\pm(z_{\sigma(1)},z_{\sigma(2)})\big(x_i^\pm(z_{\sigma(1)})x_i^\pm(z_{\sigma(2)})
    x_j^\pm(w)
   -[2]_{q^{d_{ij}}}
   x_i^\pm(z_{\sigma(1)})x_j^\pm(w)x_i^\pm(z_{\sigma(2)})\\
&&& \qquad\quad   +x_j^\pm(w)x_i^\pm(z_{\sigma(1)})x_i^\pm(z_{\sigma(2)})\big)
\Big\}\ =0,\
    \te{if }a_{ij}<0\ \te{and}\ i\notin\mathcal O(j),\\
&\te{(Q9) }&&
\sum_{\sigma\in S_3}\bigg\{p_i^\pm(z_{\sigma(1)},z_{\sigma(2)},z_{\sigma(3)})\,x_i^\pm(z_{\sigma(1)})x_i^\pm(z_{\sigma(2)})x_i^\pm(z_{\sigma(3)})
\bigg\}=0,\
    \te{if }\ s_{i}=2,
\end{align*}
where $i,j\in I$, $m,m'\in \Z^\times$, $n\in \Z$, $x^\pm_i(z)=\sum\limits_{m\in \Z} x^\pm_{i,m}z^{-m}$ and
\begin{align*}
 \phi_i^\pm(z)=q^{\pm h_{i,0}}\, \te{exp}
    \left(
        \pm (q-q\inverse)\sum\limits_{\pm m> 0}h_{i,m}z^{-m}
    \right).
\end{align*}
Here and as usual, for any $a\in \te{Span}_\C\{h_{i,0}, c\mid i\in I\}$, we used the notation
\[q^a=e^{a\hbar}.\]

When $A$ is  of finite type, the quantum algebra $\U_\hbar(\wh\fg_\mu)$ was first introduced by Drinfeld (\cite{Dr-new})
 for the purpose of providing a current algebra realization for the twisted quantum affine algebra.

Notice that the relation (Q6) is equivalent to the following relations:
\begin{align*}
&[x_{i,n}^+,x_{j,n'}^-]=\sum_{k\in\Z_N}\delta_{i,\mu^k(j)} \xi^{kn} q^{h_{j,0}} q^{\frac{n-n'}{2}c}h_{j,n+n'},\ \te{if}\ n+n'> 0,\\
   &[x_{i,n}^+,x_{j,n'}^-]=\sum_{k\in\Z_N}\delta_{i,\mu^k(j)} \xi^{kn} q^{h_{j,0}} q^{\frac{n'-n}{2}c}h_{j,n+n'},\ \te{if}\ n+n'< 0,\\
   &[x_{i,n}^+,x_{j,-n}^-]=\sum_{k\in\Z_N}\delta_{i,\mu^k(j)} \xi^{kn} (q^{nc}h_{j,0}+cq^{h_{i,0}}
   \frac{q^{nc}-q^{-nc}}{q^{c}-q^{-c}}),
\end{align*}
where $i,j\in I$ and $n,n'\in\Z$.
Then it is immediate to see that the defining relations of the classical limit  $\U_\hbar(\wh\fg_\mu)/\hbar\U_\hbar(\wh\fg_\mu)$
of $\U_\hbar(\wh\fg_\mu)$
are the same as  the defining relations of $\mathcal D_{\bm{P}}(\fg,\mu)$ with the family $\bm{P}$ defined as follows ($(i,j)\in \mathbb I$ and $\sigma\in S_2$)
\begin{equation*}
P_{ij,\sigma}(z_1,z_2,w)=\begin{cases} p_{ij}(z_1,z_2),\ &\te{if}\ i\notin \mathcal O(j),\ \sigma=1;\\
1,\ &\te{if}\ i\notin \mathcal O(j),\ \sigma=(12);\\
p_i(z_{\sigma(1)},z_{\sigma(2)},-w),\ &\te{if}\ i\in \mathcal O(j).
\end{cases}
\end{equation*}

We remark that the polynomial $p_{ij}(z,w)$ defined above coincides with that defined in \eqref{eq:def-p-ij}.
Thus,
by using Proposition \ref{D-S-R}, one can easily check that the family $\bm{P}$ defined above satisfies the conditions $(\bm{P}1)$ and $(\bm{P}2)$.
Now, as a consequence of Theorem \ref{intromain1}, one can immediately get that

\begin{thm} The classical limit $\U_\hbar(\wh\fg_\mu)/\hbar\U_\hbar(\wh\fg_\mu)$ of $\U_\hbar(\wh\fg_\mu)$ is isomorphic to the universal enveloping algebra
of $\wh\fg[\mu]$.
\end{thm}

\section{Basics on $\Gamma$-vertex algebras and their quasi-modules}
In this section we collect some basic materials on $\Gamma$-vertex algebras and their quasi-modules
needed in the proof of Theorem \ref{intromain2}.

Throughout this section, let $\Gamma$ be a fixed group equipped with a character $\phi:\Gamma\rightarrow \C^\times$.
We write $\C_\Gamma[z,w]$ for the subalgebra of the polynomial algebra $\C[z,w]$ generated by the monomials $z-\al w, \al\in \phi(\Gamma)$.
For a vector space $W$, we denote by $W((z_1,\cdots,z_s))$ the space of lower truncated (infinite) integral power series in
the commuting variables $z_1,\cdots,z_s$ with coefficients in $W$, and set $\mathcal{E}\(W\)=\te{Hom}\(W,W((z))\)$.
For each pair
\[(\bm{\al},\bm{n})=((\al_1,\cdots,\al_s),(n_1,\cdots,n_s))\in (\C^\times)^s\times \N^s,\quad s\ge 1,\] we denote that
\[\delta^{(\bm{\al},\bm{n})}(z_1,\cdots,z_s,w)=\prod_{1\le i\le s}\frac{1}{n_i!}\(\al_k^{-1}\frac{\partial}{\partial w}\)^{n_i}
z_i^{-1}\delta\(\al_k\frac{w}{z_i}\).\]
When $\bm{\al}=\bm{1}=(1,\cdots,1)$, we also write $\delta^{(\bm{n})}(z_1,\cdots,z_s,w)=\delta^{(\bm{1},\bm{n})}(z_1,\cdots,z_s,w)$.

\subsection{Definitions} In this subsection we recall the definitions of $\Gamma$-vertex algebras and their quasi-modules
introduced in \cite{Li-new-construction,Li-Gamma-quasi-mod}.

A $\Gamma$-vertex algebra is a vector space $V$ equipped with a distinguished vector $\vac$ (called the vacuum vector), a linear map
\[Y:V\rightarrow \E(V),\quad v\mapsto Y(v,z)=\sum_{n\in \Z}v_nz^{-n-1},\]
and a group homomorphism
\begin{align}\label{grouphom}
  R:\Gamma\rightarrow \te{GL}\(V\),\quad g\mapsto R_g,
\end{align}
such that the following axioms hold: for $g\in \Gamma$, $v\in V$,
\begin{align}
\label{creation} Y(\vac,x)=1_V,\quad Y(v,x)\vac\in V[[x]]\quad \te{and}\quad \lim_{z\rightarrow 0}Y(v,x)\vac=v,\\
\label{Gammava}R_g(\vac)=\vac\quad \te{and}\quad
  R_gY(v,z)R_g\inverse=Y(R_g(v),\phi(g)\inverse z),
\end{align}
and for $u,v\in V$,
\begin{align*}
&z_0^{-1}\delta\(\frac{z_1-z_2}{z_0}\)Y(u,z_1)Y(v,z_2)-z_0^{-1}\delta\(\frac{z_2-z_1}{-z_0}\)Y(v,z_2)Y(u,z_1)\\
=&\ z_2^{-1}\delta\(\frac{z_1-z_0}{z_2}\)Y(Y(u,z_0)v,z_2).
\end{align*}
In other words, a $\Gamma$-vertex algebra is a vertex algebra together with a $\Gamma$-action as in \eqref{grouphom}
such that the axiom \eqref{Gammava} holds.

A linear map $\varphi$ from a $\Gamma$-vertex algebra  $(V,Y,\vac,R)$   to another $\Gamma$-vertex algebra
$(V',Y',\vac',R')$ is called a $\Gamma$-vertex algebra homomorphism if for  $g\in \Gamma$, $v,w\in V$,
\begin{align}\varphi(\vac)=\vac',\quad \varphi(Y(v,z)w)=Y'(\varphi(v),z)\varphi(w)\quad \te{and}\quad
\varphi(R_g(v))=R'_g(\varphi(v)).
\end{align}
  The following is an analogue of \cite[Proposition 5.7.9]{LL}.
\begin{prpt}\label{gammavahom} Let $\varphi$ be a linear map from a $\Gamma$-vertex algebra $(V,\vac,Y,R)$ to a $\Gamma$-vertex algebra $(V',\vac',Y',R')$ such that
$\varphi(\vac)=\vac'$ and let $T$ be a generating subset of $V$ as a vertex algebra.
Assume that  for
 $a\in T$, $g\in \Gamma$ and $v\in V$,
\begin{align}
\label{gammahom1}\varphi(Y(a,z)v)&=Y'(\varphi(a),z)\varphi(v),\\
\label{gammahom2}\varphi(R_g(v))=R'_g(\varphi(v))&\Rightarrow \varphi(R_g(Y(a,z)v))=R_g'(Y'(\varphi(a),z)\varphi(v)).
\end{align}
Then $\varphi$ is a $\Gamma$-vertex algebra homomorphism.
\end{prpt}
\begin{proof}In view of  \cite[Proposition 5.7.9]{LL}, the condition \eqref{gammahom1} implies that
  $\varphi$ is a vertex algebra homomorphism. Thus, it remains to show that $\varphi(R_g(v))=R'_g(\varphi(v))$ for
$g\in \Gamma$, $v\in V$. This assertion can be easily checked by the assertion \eqref{gammahom1}, the condition \eqref{gammahom2} and the fact that $V$ is generated by $T$.
\end{proof}

Let $(V,Y,\vac,R)$ be a $\Gamma$-vertex algebra.
A quasi-module for $V$ is a vector space $W$ equipped with a linear map
\[Y_W: V\rightarrow \CE(W),\quad v\mapsto Y_W(v,z)\]
such that the following conditions hold: for $g\in \Gamma, v\in V$,
\begin{align}\label{quasimodule}
Y_W\(\vac,z\)=1_W,\quad  Y_W\(R_gv,z\)=Y_W\(v,\phi(g)z\),
\end{align}
and for $u,v\in V$, there exists a polynomial $0\ne f(z_1,z_2)\in \C_\Gamma[z_1,z_2]$ such that
\begin{align*}
&z_0^{-1}\delta\(\frac{z_1-z_2}{z_0}\)f(z_1,z_2)Y_W(u,z_1)Y_W(v,z_2)-
z_0^{-1}\delta\(\frac{z_2-z_1}{-z_0}\)f(z_1,z_2)\\&\cdot Y_W(v,z_2)Y_W(u,z_1)
=\ z_2^{-1}\delta\(\frac{z_1-z_0}{z_2}\)f(z_1,z_2)Y_W(Y(u,z_0)v,z_2).
\end{align*}
When $\Gamma=\{1\}$, a quasi-module for a $\Gamma$-vertex algebra is nothing but a  module for a
vertex algebra.

Finally, we remark that if $\varphi:V\rightarrow V'$ is a homomorphism of $\Gamma$-vertex algebras and $(W,Y_W')$ is
a quasi-$V'$-module, then $W$ is naturally a quasi-$V$-module with the module action $Y_W=Y_W'\circ \varphi$.

\subsection{$\Gamma$-vertex algebras arising from $\Gamma$-local subspaces}\label{subsec:Gamma-VA}
In this subsection we first recall a general construction of $\Gamma$-vertex algebras
and their quasi-modules developed in \cite{Li-new-construction}, and  then prove a result about the vanishing of certain (multi-) commutators among  vertex operators (see Lemma \ref{key}).

Let $W$ be a given vector space.
Formal series
$a(z), b(z)\in \CE(W)$ are said to be mutually $\Gamma$-local if there exists a non-zero polynomial $f(z,w)\in \C_\Gamma[z,w]$ such that
\begin{align}\label{Gammalocal}f(z,w)\ [a(z),b(w)]=0.\end{align}
In \cite{Li-new-construction}, for any mutually $\Gamma$-local pair $a(z),b(z)$ in $\CE(W)$ and $\al\in \C^\times$, the author defined a corresponding generating function
\[\CY_\al(a(z),w)b(z)=\sum_{n\in \Z}\(a(z)_{(\al,n)}b(z)\)w^{-n-1}\in \CE(W)[[w,w^{-1}]].\]
See  \cite[Definition 3.4]{Li-new-construction} for details.
 When $\al=1$, we often write
 \[\CY(a(z),w)b(z)=\CY_1(a(z),w)b(z)\quad \te{and}\quad a(z)_nb(z)=a(z)_{(1,n)}b(z).\]

A subspace $U$ of $\CE(W)$ is said to be $\Gamma$-local if any pair $a(z),b(z)\in U$ are mutually $\Gamma$-local.
For any $\al\in \C^\times$, a $\Gamma$-local subspace $U$ of $\CE(W)$ is said to be closed under $\CY_\al$ if  $a(z)_{(\al,n)}b(z)\in U$
for any pair $a(z),b(z)\in U$ and $n\in \Z$.
We define a linear map $\mathfrak R:\Gamma\rightarrow \mathrm{GL}(\E(W))$ by the rule
\begin{align}\label{Raction}\mathfrak R_g(a(z))=a(\phi(g)z),\quad g\in \Gamma,\, a(z)\in \E(W).\end{align}
The following result was proved in  \cite[Theorem 2.9]{Li-Gamma-quasi-mod}
\begin{prpt} \label{gammalocasetva}
Let $U$
be a $\Gamma$-local subspace of $\CE(W)$. Then there exists a smallest $\Gamma$-local subspace $\<U\>_\Gamma$, that contains $1_W$ and $U$ and that is closed under $\CY_\al$ for any $\al\in \phi(\Gamma)$. Moreover, $(\<U\>_\Gamma,1_W,\CY,\mathfrak R)$ is a $\Gamma$-vertex algebra and  $(W,Y_{\mathfrak W})$ is a natural faithful quasi-module
 for it,
 with \begin{align}\label{equiact1}
Y_\mathfrak W(a(z),w)=a(w),\quad a(z)\in \<U\>_\Gamma,\, w\in W.\end{align}
\end{prpt}

Let $U$ be as in Proposition \ref{gammalocasetva}.
We remark that, if $U$ is also invariant under the $\Gamma$-action \eqref{Raction}, then it
follows from \cite[Proposition 4.12]{Li-new-construction} that
\begin{align}\label{describleUgamma}
\<U\>_\Gamma=\<U\>_{\{1\}}=\te{Span}_\C\{a^{(1)}_{r_1}\cdots a^{(s)}_{r_s}1_W\mid
a^{(i)}\in U, r_i\in \Z, 1\le i\le s, s\ge 0\}.
\end{align}
Namely, $\<U\>_\Gamma$ is generated by $U$ as a vertex algebra.
We now mention a result of Li which plays an essential role  in determining the structure of various vertex algebras
arising from $\Gamma$-local sets.
\begin{lemt}\label{1-commutator} Let  $a(z), b(z)\in \CE(W)$ and let $0\ne f(z,w)\in \C_\Gamma[z,w]$.
If the $\Gamma$-locality  \eqref{Gammalocal} holds, then
\begin{align}\label{localy1}(w_1-w_2)^s\ [\CY(a(z),w_1),\CY(b(z),w_2)]=0,\end{align}
where $s$ is the order of the zero of $f(z,w)$ at $z=w$. Moreover, if
\begin{align*}
f(z,w)=\prod_{k=1}^l (z-\al_k w)^{r_k},\quad \al_k\in \phi(\Gamma),\ r_k\in \Z_+,
\end{align*}
 then one has that
\begin{align}\label{Gammalocalex}[a(z),b(w)]&=\sum_{k=1}^l\sum_{n=0}^{r_k-1} a(w)_{(\al_k,n)}b(w)\,\delta^{(\al_k,n)}(z,w),\\
\label{commutatory1}
[\CY(a(z),w_1),\CY(b(z),w_2)]&=\sum_{k=1}^l \delta_{\al_k,1}\sum_{n=0}^{r_k-1}\(\CY(a(z)_{n}b(z),w_2)\)\,\delta^{(n)}(w_1,w_2).
\end{align}
\end{lemt}
\begin{proof} The locality \eqref{localy1} was proved in  \cite[Proposition 4.8]{Li-new-construction}, the commutator formula
\eqref{Gammalocalex} was proved in \cite[Proposition 3.13]{Li-new-construction}, and
the commutator formula
\eqref{commutatory1} is impled by \eqref{localy1} and \eqref{Gammalocalex}.
\end{proof}

The rest part of this subsection is devoted to a proof of the following result.
\begin{lemt}\label{key} Let $U$ be a $\Gamma$-local subset of $\CE(W)$ and $a_1(z),\cdots,a_s(z), b(z)\in U$ for some $s\in \Z_+$.
Assume that there exist polynomials $g(z_1,\cdots,z_s,w), g'(z_1,\cdots,z_s,w)$ in $\C[z_1,\cdots,z_s,w]$
such that
\begin{align}\label{keyc1}
g(w,\cdots,w,w)\ne 0,\end{align} and such that
\begin{align}\label{keyc2}  g(z_1,\cdots,z_s,w)g'(z_1,\cdots,z_s,w)\, [a_1(z_1),\cdots, [a_s(z_s), b(w)]]=0.
\end{align}
Then the following relation holds true
\begin{align}g'(w_1,\cdots,w_s,w)\,[\CY(a_1(z),w_1),\cdots, [\CY(a_s(z),w_s), \CY(b(z),w)]]=0.
\end{align}
\end{lemt}

Before proving Lemma \ref{key}, we present two of its by-products, which will play a key role in our proof of Theorem \ref{intromain2}.
\begin{cort}\label{keycor0}
 Let $U$ be a $\Gamma$-local subset of $\CE(W)$ and $a_1(z),a_2(z), b(z)\in U$.
 Assume that there exists a positive integer $M$ such that
\begin{align*}
(z_1^M-z_2^M)\,[a(z_1),[a(z_2),b(w)]]=0.
\end{align*}
Then the following relation holds true
\begin{align*}(w_1-w_2)\,[\CY(a_1(z),w_1),[\CY(a_2(z),w_{2}), \CY(b(z),w)]]=0.
\end{align*}
\end{cort}
\begin{proof} By applying Lemma \ref{key} with $g(z_1,z_2,w)=z_1-z_2$ and $g'(z_1,z_2,w)=\frac{z_1^M-z_2^M}{z_1-z_2}$.
\end{proof}
\begin{cort}\label{keycor}
 Let $U$ be a $\Gamma$-local subset of $\CE(W)$ and $a_1(z),\cdots,a_s(z), b(z)\in U$ for some $s\in \Z_+$.
Assume that there exist polynomials $h_{ij}(z,w)\in \C_\Gamma[z,w], 1\le i<j\le s$ and
 $g_\sigma(z_1,\cdots,z_s,w)\in \C[z_1,\cdots,z_s,w], \sigma\in S_s$
such that
\begin{align}\label{keyc3}
h_{ij}(w,w)\ne 0,\quad h_{ij}(z,w)\,[a_i(z),a_j(w)]=0,\ i,j=1,\cdots,s,
\end{align}
and moreover 
\begin{align}
\label{keyc4}&\sum_{\sigma\in S_s}g_\sigma(w,\cdots,w,w)\ne 0,\\
 \label{keyc5}&\sum_{\sigma\in S_s} g_\sigma(z_1,\cdots,z_s,w)\, [a_1(z_{\sigma(1)}),\cdots, [a_s(z_{\sigma(s)}), b(w)]]=0.
\end{align}
Then the following relation holds true
\begin{align}[\CY(a_1(z),w_1),\cdots,[\CY(a_{s-1}(z),w_{s-1}), [\CY(a_s(z),w_s), \CY(b(z),w)]]]=0.
\end{align}
\end{cort}
\begin{proof} We need to introduce the polynomials
\begin{align*} h(z_1,\cdots,z_s)&=\prod_{1\le i<j\le s}\prod_{1\le a<b\le s} h_{ij}(z_a,z_b),\\
g(z_1,\cdots,z_s,w)&=h(z_1,\cdots,z_s)\cdot(\sum_{\sigma\in S_s} g_\sigma(z_1,\cdots,z_s,w)).
\end{align*}
From \eqref{keyc3} and \eqref{keyc4}, it follows  that
\begin{align}\label{keyc6} g(w,\cdots,w,w)\ne 0.
\end{align}
Moreover, by using \eqref{keyc3} and the Jacobi identity, one has  that
\begin{align*}
h(z_1,\cdots,z_s)\([a_1(z_1),\cdots, [a_s(z_s), b(z)]]-[a_1(z_{\sigma(1)}),\cdots, [a_s(z_{\sigma(s)}), b(z)]]\)=0,
\end{align*}
for all $\sigma\in S_s$.
This together with \eqref{keyc5} gives that
\begin{align} \label{keyc7}
g(z_1,\cdots,z_s,w)\,[a_1(z_1),\cdots, [a_s(z_s), b(z)]]=0.
\end{align}
So the assertion is implied by Lemma \ref{key} (with $g'(z_1,\cdots,z_s,w)=1)$, \eqref{keyc6} and \eqref{keyc7}.
\end{proof}

Now we turn to prove Lemma \ref{key}. We start with a slight generalization of Lemma \ref{1-commutator}.

\begin{lemt}\label{s-commutator}  Let $U$ be a $\Gamma$-local subset of $\CE(W)$ and $a_1(z),\cdots,a_s(z), b(z)\in U$ for some $s\in \Z_+$.
Then the formal series $[a_1(z_1),\cdots,[a_s(z_s), b(z)]]$ is a finite summation of the form
\begin{align}\label{s-commutator1}
\sum_{(\bm{\al},\bm{n})\in (\phi(\Gamma))^s\times \N^s} c_{\bm{\al},\bm{n}}(w)\ \delta^{(\bm{\al},\bm{n})}(z_1,\cdots,z_s,z)
\end{align}
for some uniquely determined formal series $c_{\bm{\al},\bm{n}}(z)\in \<U\>_\Gamma$.
Moreover,
\[[\CY(a_1(z),w_1),\cdots, [\CY(a_s(z),w_s), \CY(b(z),w)]]\]
is a finite summation of the form
\begin{align}\label{s-commutator2} \sum_{(\bm{\al},\bm{n})\in (\phi(\Gamma))^s\times \N^s}\delta_{\bm{\al},\bm{1}}\, \CY(c_{\bm{1},\bm{n}}(z),w)\ \delta^{(\bm{n})}
 (w_1,\cdots,w_s,w).\end{align}
 \end{lemt}
 \begin{proof}When $s=1$, the assertion is implied by Lemma \ref{1-commutator}.
 For the general case, one can prove the assertion by
 using   Lemma \ref{1-commutator} and  an induction argument. We omit the details.
 \end{proof}
We now record two well-known properties of $\delta$-functions for later use, whose proofs can be found respectively in
 \cite[\S\,2]{Li-local-twisted} and \cite[(2.3.51)]{LL}.
\begin{lemt}\label{delta1}  Let $f_1(w),\cdots,f_l(w)\in W[[w,w^{-1}]]$ and
let $(\bm{\al}_1,\bm{n}_1),\cdots,
(\bm{\al}_l,\bm{n}_l)$ be some distinct pairs in $(\phi(\Gamma))^s\times \N^s$, where $s\in \Z_+$.
Then
\[\sum_{k=1}^l f_{k}(w)\,\delta^{(\bm{\al}_k,\bm{n}_k)}(z_1,\cdots,z_s,w)=0\] if and only if
$f_{k}(w)=0$ for all $k=1,\cdots,l$.
\end{lemt}

\begin{lemt}\label{delta2}
For any Laurent polynomial $f(w_1,w_2)$ and $\al\in \C^\times$, one has that
\begin{align*}
&f(w_1,w_2)\, \(\frac{\partial}{\partial w_1}\)^{n}\delta\(\al\frac{w_1}{w_2}\)\\
=&\ \sum_{k=0}^n(-1)^k {n\choose k}\(\(\frac{\partial}{\partial w_1}\)^{k}f\)(w_2,\al w_2) \(\frac{\partial}{\partial w_1}\)^{n-k}\delta\(\al \frac{w_1}{w_2}\).\end{align*}
\end{lemt}

 Here we are ready to finish the proof of Lemma \ref{key}. Let $a_1(z),\cdots,a_s(z)$, $b(z)$ and $g(z_1,\cdots,z_s,w),
 g'(z_1,\cdots,z_s,w)$ be as in
 Lemma \ref{key}. Due to Lemma \ref{s-commutator}, we may write the the commutators
 \[[a_1(z_1),\cdots,[a_s(z_s),b(w)]]\quad \te{and}\quad [\CY(a_1(z),w_1),\cdots, [\CY(a_s(z),w_s), \CY(b(z),w)]]\]
as in \eqref{s-commutator1} and \eqref{s-commutator2}, respectively.
Given an  $(\bm{\al},\bm{n})\in (\phi(\Gamma))^s\times \N^s$.
Then it follows from Lemma \ref{delta2} that  the formal series
\[A(z_1,\cdots,z_s,w)\, \delta^{(\bm{\al},\bm{n})}(z_1,\cdots,z_s,w)\quad (A=g\ \te{or}\ g')\]
is a finite summation of the form
\begin{align}\label{keyproof1}
\sum_{\bm{m}\preccurlyeq \bm{n}\in \N^s} A_{\bm{\al},\bm{n},\bm{m}}(w)\, \delta^{(\bm{\al},\bm{m})}(z_1,\cdots,z_s,w)\end{align}
for some $A_{\bm{\al},\bm{n},\bm{m}}(w)\in \C[w]$, where $\bm{m}\preccurlyeq \bm{n}$ means that $m_k\le n_k$ for all $1\le k\le s$.

For any $\bm{m}\preccurlyeq \bm{n}$, set
\begin{align*}
c_{\bm{\al},\bm{n},\bm{m}}(w)&=g'_{\bm{\al},\bm{n},\bm{m}}(w)\,c_{\bm{\al},\bm{n}}(w)\in W[[w,w^{-1}]],\\
b_{\bm{n},\bm{m}}(w)&=g'_{\bm{1},\bm{n},\bm{m}}(w)\,\CY(c_{\bm{1},\bm{n}}(z),w)\in \CE(W)[[w,w^{-1}]].\end{align*}
Then one can conclude from \eqref{s-commutator1} and \eqref{keyproof1} (with $A=g')$ that the formal series
\[g'(z_1,\cdots,z_s,w)\, [a_1(z_1),\cdots,[a_s(z_s),b(w)]]\]
is a finite summation of the form
\begin{align}\label{firstexp}\sum_{(\bm{\al},\bm{n})\in (\phi(\Gamma))^s\times \N^s} c'_{\bm{\al},\bm{n}}(w)\,\delta^{(\bm{\al},\bm{n})}(z_1,\cdots,z_s,w),
\end{align}
where $c'_{\bm{\al},\bm{n}}(w)=\sum_{\bm{n}'\succcurlyeq \bm{n}} c_{\bm{\al,\bm{n}',\bm{n}}}(w)$.
Similarly,  by applying \eqref{s-commutator2} and \eqref{keyproof1} (with $A=g')$, we know that the formal series
 \[g'(w_1,\cdots,w_s,w)\,[\CY(a_1(z),w_1),\cdots, [\CY(a_s(z),w_s), \CY(b(z),w)]]\]
 has the following expression
 \[\sum_{\bm{n}\in \N^s} b_{\bm{n}}'(w)\, \delta^{(\bm{n})}
 (w_1,\cdots,w_s,w),\]
 where $b_{\bm{n}}'(w)=\sum_{\bm{n}'\succcurlyeq \bm{n}} b_{\bm{\bm{n}',\bm{n}}}(w)$.

 Thus, it suffices to prove that $b'_{\bm{n}}(w)=0$ for all $\bm{n}\in \N^s$. Indeed, assume conversely that the finite set
 $\{\bm{n}\in \N^s\mid b'_{\bm{n}}(w)\ne 0\}$ is non-empty. Let us take a maximal element $\bm{n}_0$  in this finite
set  with respect to the partial order $\preccurlyeq$.
We remark that for $\bm{\al}\in \phi(\Gamma)$ and $\bm{m},\bm{n}\in \N^s$ with
 $\bm{m}\preccurlyeq\bm{n}$,
 the formal series $c_{\al,\bm{n},\bm{m}}(z)$ may be not contained in  $\<U\>_\Gamma$.
 But these formal series belong to any maximal $\Gamma$-local subspace, say $U'$, of $\CE(W)$ that contains $1_W$ and $U$.
 Moreover, as operators on $\mathrm{End}(U')[[w,w^{-1}]]$, one has by definition (\cite[Definition 3.4]{Li-new-construction}) that
\[\CY(c_{\bm{1},\bm{n},\bm{m}}(z),w)=b_{\bm{n},\bm{m}}(w).\]
This in particular shows that $c'_{\bm{1},\bm{n}_0}(w)\ne 0$ and that $\bm{n}_0$ is also maximal in the set
$\{\bm{n}\in \N^s\mid c'_{\bm{n}}(w)\ne 0\}$.
Now, from \eqref{firstexp} and \eqref{keyproof1} (with $A=g)$, we know that the formal series
\[g(z_1,\cdots,z_s,w)g'(z_1,\cdots,z_s,w)\, [a_1(z_1),\cdots,[a_s(z_s),b(w)]]\]
is a finite summation of the form
\begin{align}\label{firstexp}\sum_{(\bm{\al},\bm{n})\in (\phi(\Gamma))^s\times \N^s} c''_{\bm{\al},\bm{n}}(w)\,\delta^{(\bm{\al},\bm{n})}(z_1,\cdots,z_s,w),
\end{align}
where $c''_{\bm{\al},\bm{n}}(w)=\sum_{\bm{n}'\succcurlyeq \bm{n}} g_{\bm{\al},\bm{n'},\bm{n}}(w)\,c'_{\bm{\al},\bm{n}'}(w)$.
This together with Lemma \ref{delta1} and the assumption \eqref{keyc2} gives that
\begin{align}\label{cprime}
c''_{\bm{\al},\bm{n}}(w)=0,\quad \forall\ (\bm{\al},\bm{n})\in (\phi(\Gamma))^s\times \N^s.\end{align}
But, by the maximality of $\bm{n}_0$ and the assumption \eqref{keyc1}, one has that
\[c''_{\bm{1},\bm{n}_0}(w)=g_{\bm{1},\bm{n}_0,\bm{n}_0}(w)\,c'_{\bm{1},\bm{n}_0}(w)=g(w,w,\cdots,w)\,c'_{1,\bm{n}_0}(w)\ne 0,\]
a contradiction to \eqref{cprime}. This finishes the proof of Lemma \ref{key}.

\subsection{$\Gamma$-conformal Lie algebras and their universal  $\Gamma$-vertex algebras}In this subsection
we recall another general construction of $\Gamma$-vertex algebras and their quasi-modules given in \cite{Li-Gamma-quasi-mod}.

Recall that a conformal Lie algebra $(C,Y_-,T)$ (\cite{Kac-VA}), also known as a vertex Lie algebra (\cite{DLM,Primc-VA-gen-by-Lie}),
is a vector space $C$ equipped with a linear operator $T$ and a linear map
 \begin{align}\label{Y-} Y_-: C\rightarrow \mathrm{Hom}(C, z^{-1}C[z^{-1}]),\quad a\mapsto Y_-(u,z)=\sum_{n\ge 0}u_{(n)} z^{-n-1}\end{align}
 such that for any $u,v\in C$,
 \begin{align}
\label{proT} &[T, Y_-(u,z)]=Y_-(Tu,z)=\frac{\rd}{\rd z}Y_-(u,z),\\
  &Y_-(u,z)v=\mathrm{Sing}\(e^{zT}Y_-(v,-z)u\),\notag\\
 &[Y_-(u,z),Y_-(v,w)]=\mathrm{Sing}(Y_-(Y_-(u,z-w)v,w)),\notag\end{align} where
 $\mathrm{Sing}$ stands for the singular part.

 A conformal Lie algebra structure on a vector space $C$ exactly amounts to a Lie algebra structure on the quotient
space
\[\wh{C}=\C[t,t^{-1}]\ot C/(1\ot T+\frac{\rd}{\rd t}\ot 1)(\C[t,t^{-1}]\ot C)\] of $\C[t,t^{-1}]\ot C$.
Let us denote by
\[\rho: \C[t,t^{-1}]\ot C\rightarrow \wh{C},\quad t^m\ot u\mapsto u(m),\quad u\in C,\ m\in \Z\] the natural quotient map.
The following result  was proved in \cite[Remark 4.2]{Primc-VA-gen-by-Lie}.
 \begin{lemt}\label{lem:determineC} Let $C$ be a vector space equipped with a linear operator $T$ and a linear map
 $Y_-$ as in \eqref{Y-} such that \eqref{proT} holds.
 Then $C$ is a conformal Lie algebra if and only if there is a Lie algebra structure on $\wh{C}$
 such that
 \begin{align}\label{LieC}[u(m), v(n)]=\sum_{i\ge 0}{{m}\choose{i}}(u_{(i)}v)(m+n-i),\end{align}
for $u,v\in C,\ m,n\in \Z$.
 \end{lemt}

Let $C$ be a conformal Lie algebra.  Set
\[\wh{C}^{(-)}=\rho(t^{-1}\C[t^{-1}]\ot C)\quad\te{and}\quad \wh{C}^{(+)}=\rho(\C[t]\ot C).\] Then both $\wh{C}^{(+)}$
and $\wh{C}^{(-)}$ are subalgebras of $\wh{C}$ and
\begin{align}\label{polardec}
\wh{C}=\wh{C}^{(+)}\oplus \wh{C}^{(-)}
\end{align} is a polar
 decomposition of $\wh{C}$. Moreover,   the map
 \begin{align}\label{iso-1}
 C\rightarrow \wh{C}^{(-)},\quad u\mapsto u(-1)\end{align} is an isomorphism of vector spaces (\cite[Theorem 4.6]{Primc-VA-gen-by-Lie}).

 Consider the induced
$\wh{C}$-module \[V_C=\U(\wh{C})\otimes_{\U(\wh{C}^{(+)})}\C,\] where $\C$ is the one dimensional trivial $\wh{C}^+$-module.
Set $\vac=1\ot 1$. Identify $C$ as a subspace of $V_C$ through the linear map $u\mapsto u(-1)\vac$.
Then it was proved in \cite{Primc-VA-gen-by-Lie} that there exists a
unique vertex algebra structure on $V_C$  with $\vac$ as the vacuum vector and with
\[Y(u,z)=u(z)=\sum_{n\in \Z}u(n)z^{-n-1}\]
 for $u\in C$.
In the literature, $V_C$ is often called the universal  vertex algebra associated to $C$.

 \begin{lemt}\label{C0genVC}
 Let $C_0$ be a subset of  $C$ such that $\rho(\C[t,t^{-1}]\ot C_0)$ generates
the associated Lie algebra $\wh{C}$. Then $C_0$
 generates $V_C$ as a vertex algebra.
\end{lemt}
\begin{proof}Since $C$ generates $V_C$ as a vertex algebra, we only need to show that
$C_0$ generates $C$ as a conformal Lie algebra.
 Using \eqref{LieC}, one knows that any element in the subalgebra of $\wh{C}$
 generated by $\rho(\C[t,t^{-1}]\ot C_0)$ is a finite summation
$\sum u^{(i)}(n_i)$
for some $u^{(i)}\in \<C_0\>$, $n_i\in \Z$, where  $\<C_0\>$ indicates the conformal Lie subalgebra of $C$ generated by $C_0$.
As $\rho(\C[t,t^{-1}]\ot C_0)$ generates the Lie algebra $\wh{C}$, for any $u\in C$,
there exist finitely many $u^{(i)}\in \<C_0\>$ and $n_i\in \Z$ such that
\begin{align}\label{u-1}u(-1)=\sum u^{(i)}(n_i).\end{align}
Note that $\wh{C}^{(+)}\cap \wh{C}^{(-)}=0$ and $k! u(-k-1)=T^k(u)(-1)$ for $u\in C$, $k\in \Z_+$.
Thus, we may (and do) assume that all the integers $n_i$ appeared in \eqref{u-1} are $-1$.
Then, due to the isomorphism given in \eqref{iso-1}, one gets that $\<C_0\>=C$, as desired.
\end{proof}

Recall from \cite{Li-Gamma-quasi-mod} (see also \cite{GKK-Gamma-conformal-alg}) that a $\Gamma$-conformal Lie algebra $(C,Y_-,T,R)$ is a conformal Lie algebra $(C,Y_-,T)$ equipped with a group homomorphism
\[R: \Gamma\rightarrow \mathrm{GL}(C),\quad g\mapsto R_g\] such that for any $u, v\in C$,
\begin{align}
\label{proTR}&T R_g=\phi(g)R_gT,\\
&\label{proRYR} R_gY_-(u,z)R_{g^{-1}}=Y_-(R_gu, \phi(g)^{-1}z)),\\
&Y_-(R_gu,z)v=0\quad \te{for all but finitely many }\ g\in \Gamma.\notag \end{align}

 As was pointed out in \cite{GKK-Gamma-conformal-alg}, the following result shows that
 a $\Gamma$-conformal Lie algebra exactly amounts to a conformal Lie algebra equipped
 with a $\Gamma$-action on the associated Lie algebra  by automorphisms.

\begin{lemt}\label{whRg} Let $(C,Y_-,T)$ be a conformal Lie algebra and let $R: \Gamma\rightarrow \mathrm{GL}(C)$ be a group homomorphism such that \eqref{proTR} holds.
Then $(C,Y_-,T,R)$ is a $\Gamma$-conformal Lie algebra if and only if for each $g\in \Gamma$, the map
\begin{align}\label{defwhRg}\wh{R}_{g}: \wh{C}\rightarrow \wh{C};\quad u(m)\mapsto \phi(g)^{m+1}(R_g(u)(m)),\quad u\in C,\ m\in \Z \end{align}
is an automorphism of $\wh{C}$ and for each $u,v\in C$, $[\wh{R}_{g}(u)(z), v(w)]=0$ for all but finitely many $g\in \Gamma$.
\end{lemt}
\begin{proof} Using \eqref{proTR}, one can easily check that the map $\wh{R}_{g}$ is well-defined. Moreover, for any $u,v\in C$ and $g\in \Gamma$,
\begin{align*}[\wh{R}_{g}(u)(z), \wh{R}_{g}(v)(w)]=\,&[R_g(u)(\phi(g)^{-1}z), R_g(v)(\phi(g)^{-1}w)]\\
=&\sum_{i\ge 0}\phi(g)^{i+1}\(R_g(u)_{(i)}R_g(v)\)(\phi(g)^{-1}w)\delta^{(i)}(z,w).
\end{align*}
On the other hand,
\begin{align*}
\wh{R}_{g}([u(z),v(w)])=\sum_{i\ge 0} \wh{R}_{g}(u_{(i)}v)(w)\delta^{(i)}(z,w)=\sum_{i\ge 0} R_g(u_{(i)}v)(\phi(g)^{-1}w)\delta^{(i)}(z,w).
\end{align*}
Thus, $\wh{R}_{g}$ is an automorphism of $\wh{C}$ if and only if for any $i\in \Z$,
\[R_g(u_{(i)}v)=\phi(g)^{i+1}(R_gu)_{(i)}(R_gv),\]
Notice that this identity is equivalent to that in \eqref{proRYR} and so we complete the proof of lemma.
\end{proof}

Suppose now that $C$ is a $\Gamma$-conformal Lie algebra. Then  for each $g\in \Gamma$,  the automorphism $\wh{R}_{g}$ on $\wh{C}$ preserves  the polar decomposition \eqref{polardec}
and hence induces a linear automorphism, still denoted by $R_g$, on $V_C\cong \U(\wh{C}^{(-)})$.
The following result was proved in  \cite[Lemma 4.16]{Li-Gamma-quasi-mod}.
\begin{lemt}\label{gammacontova} If  $C$ is a $\Gamma$-conformal Lie algebra, then the universal vertex algebra $V_C$ associated to $C$
 is a $\Gamma$-vertex algebra with
the $\Gamma$-action given by
 \[R:\Gamma\rightarrow \mathrm{GL}(V_C),\quad g\mapsto R_g,\ g\in \Gamma.\]
 \end{lemt}

Following \cite{Li-Gamma-quasi-mod}, we define a new operation $[\cdot,\cdot]_\Gamma$ on $\wh{C}$ by letting
\[[a, b]_{\Gamma}=\sum_{g\in \Gamma}[\wh{R}_{g}(a),b]\]
for $a,b\in \wh{C}$.
Under this operation, the quotient space
\[\wh{C}_\Gamma=\wh{C}/\te{Span}_\C\{\wh{R}_{g}(a)-a\mid a\in\wh{C}, g\in \Gamma\}\]
becomes a Lie algebra (\cite[Lemma 4.1]{Li-Gamma-quasi-mod}).
\begin{remt}{\em \label{whCGamma} If $\Gamma=\<g\>$ is a finite cyclic group, then it is easy to see that
$\wh{C}_\Gamma$ is isomorphic to the  subalgebra of
 $\wh{C}$ fixed by $\wh{g}$.
  }\end{remt}

For any $u\in C$, $m\in \Z$, we will also denote the image of $u(m)$  in $\wh{C}_\Gamma$ by itself.
Note that there is a natural $\Z$-grading structure on $\wh{C}_\Gamma$ such that
$\deg(u(m))=m$ for $u\in C, m\in \Z$.
For any restricted $\wh{C}_\Gamma$-module $W$ and $u\in C$, we set
\[u(z)=\sum_{m\in \Z}u(m)z^{-m-1}\in \CE(W).\] The following result was proved in \cite[Theorem 4.17]{Li-Gamma-quasi-mod}.

\begin{prpt}\label{quasitores} Let  $C$ be a $\Gamma$-conformal Lie algebra. Assume that the character $\phi:\Gamma\rightarrow \C^\times$ is injective. Then any restricted module $W$ for $\wh{C}_\Gamma$ is naturally a quasi-module
for the $\Gamma$-vertex  algebra $V_C$ with $Y_W(u,z)=u(z)$ for $u\in C$. On the other hand, any quasi-module $W$ for the
$\Gamma$-vertex algebra $V_C$ is naturally a restricted module for the Lie algebra $\wh{C}_\Gamma$ with $u(z)=Y_W(u,z)$ for $u\in C$.
\end{prpt}

\section{Proof of Theorem \ref{intromain1} and Theorem \ref{intromain2}}

In this section we prove the main results of this paper (Theorems \ref{intromain1} and \ref{intromain2}) stated in the Introduction.
Throughout this section, let
$\Gamma=\<\mu\>$ and let
$\phi:\Gamma\rightarrow \C^\times$ be the character determined by the rule $\phi(\mu)=\xi^{-1}$.
\subsection{Proof of  Theorem \ref{intromain2}}\label{subsec:proof-thm-intromain2}
As in Theorem \ref{intromain2}, let $\fm$ be one of the algebras $\fg, \fn_+$ and $\fn_-$, and
$W$ an arbitrary restricted $\mathcal D_{\bm{P}}(\fm,\mu)$-module.
 We define a subspace $U_{\fm,W}$ of $\CE(W)$ as follows
\[U_{\fn_+,W}=\sum_{i\in I}\C x_i^+(z),\ U_{\fn_-,W}=\sum_{i\in I}\C x_i^-(z)\ \te{and}\
U_{\fg,W}=U_{\fn_+,W}\oplus U_{\fn_-,W}.\]
Then one can conclude from the defining relation $(X\pm)$ of $\mathcal D_{\bm{P}}(\fm,\mu)$ that $U_{\fm,W}$ is a $\Gamma$-local subspace of $\CE(W)$.
Therefore, it follows from Proposition \ref{gammalocasetva} that  $U_{\fm,W}$ generates a $\Gamma$-vertex algebra $(\<U_{\fm,W}\>_\Gamma,1_W,\CY,\mathfrak R)$
 and that $(W,Y_{\mathfrak W})$ is a quasi-module of
$\<U_{\fm,W}\>_\Gamma$ with
\begin{align}\label{actionyW}
Y_{\mathfrak W}(x_i^\pm(z),w)=x_i^\pm(w),\quad i\in I.
\end{align}
Moreover, it follows form \eqref{Raction} and the first relation in $(X\pm)$  that
\begin{align}\label{Ractiononx}
\mathfrak R_{\mu^k}(x_i^\pm(z))=x_i^\pm(\xi^{-k}z)=\xi^k\, x_{\mu(i)}^\pm(z),\quad i\in I,\, k\in \Z_N.\end{align}

When $\mu=1$, recall the usual Serre relation $(DS)_{\bm{p}}$ given in \eqref{depun}.
The following result will be proved in  \S\,\ref{subsec:proof-prop-D(m,1)}.
\begin{prpt}\label{reducedtoun}  $\mathcal D_{\bm{p}}(\fg,1)$
is a Drinfeld type presentation of $\wh\fg$.
\end{prpt}

Assume now that  Proposition \ref{reducedtoun} holds. Then we can prove the following result.
\begin{lemt}\label{moduleUW}Let $\fm$ be one of the algebras $\fg, \fn_+$ and $\fn_-$.
 Then there is an $\wh\fm$-module structure on the $\Gamma$-vertex algebra $\< U_{\fm,W}\>_\Gamma$ with the action
 determined by
 \begin{align}\label{UWmodact}e_i^\pm(w)=\CY(x_i^\pm(z),w),\quad i\in I.\end{align}
Furthermore, as an $\wh\fm$-module, $\< U_{\fm,W}\>_\Gamma$ is generated by $1_W$ and the elements  in
 \begin{align}\label{genm+}\{t_1^m\ot e_i^\pm\mid i\in I, m\in \N\}\cap \wh\fm\end{align} act trivially on $1_W$.
\end{lemt}
\begin{proof}We first prove that the action \eqref{UWmodact} determines an $\wh\fm$-module structure on $\< U_{\fm,W}\>_\Gamma$
with $\fm=\fn_+$.
In view of  Proposition \ref{reducedtoun}, it suffices to show that the
operators $\CY(x_i^+(z),w)$, $i\in I$ satisfy the  relations
\begin{align}\label{untwistedre1}&f_{ij}(w_1-w_2)\cdot[\CY(x_i^+(z),w_1),\CY(x_j^+(z),w_2)]=0,\ \te{for}\ i,j\in I,\\
\label{untwistedre3}&[\CY(x_i^+(z),w_1),\cdots,[\CY(x_i^+(z),w_{1-a_{ij}}),\CY(x_j^+(z),w)]]=0,\ \te{for}\ (i,j)\in \mathbb I.
\end{align}
and if $\fg$ is of type $A_1^{(1)}$, satisfy the following addition relation
\begin{align}
\label{untwistedre2}(w_1-w_2)\,[\CY(x_i^+(z),w_1),[\CY(x_i^+(z),w_{2}),\CY(x_j^+(z),w)]]=0,\ \te{for}\ i\ne j.
\end{align}
Indeed, the relation \eqref{untwistedre1} is implied by the defining relation $(X+)$ of
$\mathcal D_{\bm{P}}(\fn_+,\mu)$ and \eqref{localy1},
the relation \eqref{untwistedre3} follows from the defining relation $(DS+)_{\bm{P}}$  of
$\mathcal D_{\bm{P}}(\fn_+,\mu)$, the condition $(\bm{P}2)$ and
Corollary \ref{keycor}, and the relation \eqref{untwistedre2} is implied by the defining relation $(AS\pm)$ of
$\mathcal D_{\bm{P}}(\fn_+,\mu)$ and Corollary
\ref{keycor0}.
Similarly, one can prove that $\< U_{\fn_-,W}\>_\Gamma$ is an $\wh\fn_-$-module with the action determined by \eqref{UWmodact}.

Now we turn to consider the case that $\fm=\fg$. In this case the formal series
\[h_i(z)=\sum_{m\in \Z} h_{i,m}z^{-m-1},\ i\in I, \ c(z)=\sum_{m\in \Z}\delta_{m,-1}c z^{-m-1}\]
are also contained in  $\< U_{\fg,W}\>_\Gamma$.
We are going to prove  that the action
\[e_i^\pm(w)=\CY(x_i^\pm(z),w),\ \al_i^\vee(w)=\CY(h_i(z),w),\ i\in I,\
\rk_1=N\CY(c(z),w)\]
determines a  $\wh\fg$-module structure on $\< U_{\fg,W}\>_\Gamma$. This action is well-defined as
\[\frac{\partial}{\partial z}\CY(c(z),w)=\CY(\frac{\partial}{\partial z}c(z),w)=0.\]
Again by using Proposition \ref{reducedtoun}, in addition the relations in \eqref{untwistedre1}\ and  \eqref{untwistedre3},
it remains to prove that the following commutation relations:
\begin{align}
\label{untwistedre4}[\CY(h_i(z),w_1),\CY(h_j(z),w_2)]=&\ \epsilon_j^{-1}a_{ij}\,\CY(c(z),w_2)\delta^{(1)}(w_1,w_2);\\
\label{untwistedre5}[\CY(h_i(z),w_1),\CY(x_j^\pm(z),w_2)]=& \pm a_{ij}\,\CY(x_j^\pm(z),w_2)\delta^{(0)}(w_1,w_2);\\
\label{untwistedre6} [\CY(h_i(z),w_1),\CY(c(z),w_2)]=&\, 0 =\, [\CY(x_i^\pm(z),w_1),\CY(c(z),w_2)];\\
\label{untwistedre7}[\CY(x_i^+(z),w_1),\CY(x_j^-(z),w_2)]=&\ \delta_{i,j}\big(\CY(h_j(z),w_2)\delta^{(0)}(w_1,w_2)\\
\notag&+\epsilon_j^{-1}\,\CY(c(z),w_2)\delta^{(1)}(w_1,w_2)\big),
\end{align}
for $i,j\in I$. Now we rewritten the defining relation $(H)$ of $\mathcal D_{\bm{P}}(\fg,\mu)$ as follows
\begin{align*} [h_i(z), c(w)]=0,\
[h_i(z),h_j(w)]=\sum_{k\in \Z_N}\frac{N}{\epsilon_j}a_{i\mu^k(j)}c(w_2)\delta^{(\xi^k,1)}(w_1,w_2).
\end{align*}
This together with Lemma \ref{1-commutator} gives  \eqref{untwistedre4}.
The relations \eqref{untwistedre5}-\eqref{untwistedre7} can be proved in a similar way by using
the remaining defining relations of $\mathcal D_{\bm{P}}(\fg,\mu)$ and Lemma \ref{1-commutator}, and
we omit the details.

For the furthermore statement, it follows from \eqref{Ractiononx} that $U_{\fm,W}$ is a $\Gamma$-submodule of $\CE(W)$
and hence $\<U_{\fm,W}\>_\Gamma=\<U_{\fm,W}\>_{\{1\}}$ (see \eqref{describleUgamma}).
 This implies that the
$\wh\fm$-module $\<U_{\fm,W}\>_\Gamma$ is generated by $1_W$. So it remains to prove that the elements in \eqref{genm+} act trivially on $1_W$, or equivalently,
to prove that
\[\CY(x_i^\pm(z),w).1_W\in \<U_{\fm,W}\>_\Gamma[[w]].\]
But this is just the creation property of the vacuum vector $1_W$ (see \eqref{creation}), and so we complete the proof.
\end{proof}

Let   $\wh\fm$ be one of the algebras $\wh\fg,
\wh\fn_+$ and $\wh\fn_-$. We define
$\wh\fm^{(+)}$ to be the subalgebra of $\wh\fm$ generated by the elements in \eqref{genm+}.
Form
the induced $\wh{\mathfrak m}$-module
\[V(\wh{\mathfrak m})= \U\(\wh{\mathfrak m}\)\otimes_{\U\(\wh{\mathfrak m}^{(+)}\)} \C,\]
where $\C$ stands for the one dimensional trivial $\wh{\mathfrak m}^{(+)}$-module.
Then it follows from Lemma \ref{moduleUW} and the universality of the $\wh\fm$-module $V(\wh\fm)$ that there is  a (unique) surjective $\wh\fm$-module homomorphism
\[\varphi_{\fm,W}:V(\wh\fm)\rightarrow \<U_{\fm,W}\>_\Gamma\] such that $\varphi_{\fm,W}(\vac)=1_W$.
In what follows, we will often  view
\begin{align}\label{defTfm}
T_\fm=\{e_i^\pm\mid i\in I\}\cap \fm\end{align} as a subset of $V(\wh{\fm})$  in sense of that
\[
e_i^\pm=(t_1^{-1}\ot e_i^\pm)\ot 1,\quad i\in I.\]
Since $\varphi_{\fm,W}$ is an $\wh\fm$-module homomorphism, one gets that
\[\varphi_{\fm,W}(e_i^\pm(w).\vac)=e_i^\pm(w).\varphi_{\fm,W}(\vac)=\CY(x_i^\pm(z),w)1_W,\quad i\in I.\]
 By  the creation property on the vacuum vector  (see \eqref{creation}), this implies that
\begin{align}\label{eixi} \varphi_{\fm,W}(e_i^\pm)=x_i^\pm(z),\quad i\in I.
\end{align}

 The following result will be proved in \S\,\ref{subsec:proof-prop-universalva}.

\begin{prpt}\label{universalva} Let $\wh\fm$ be one of the algebras $\wh\fg,
\wh\fn_+$ and $\wh\fn_-$.
Then there exists a $\Gamma$-vertex algebra structure on the induced $\wh\fm$-module $V(\wh\fm)$
such that $\vac=1\ot 1$ is the vacuum vector and such that $T_\fm$ generates $V(\wh\fm)$ as a vertex algebra.
The vertex operators on $T_\fm$ are given by
\begin{align}\label{uvavo}
Y(e_i^\pm,z)=e_i^\pm(z)=\sum_{m\in\Z}(t_1^m\ot e_i^\pm)z^{-m-1},\quad i\in I,\end{align}
and the $\Gamma$-action on $T_\fm$ is determined by
 \begin{align}\label{uvar}
 R_\mu(e_i^\pm)=\xi^{-1} e_{\mu(i)}^\pm,\quad i\in I.
 \end{align}
 Moreover, any quasi-module $W$ for the
$\Gamma$-vertex algebra $V(\wh\fm)$
  is naturally a restricted module
for the Lie algebra $\wh\fm[\mu]$ with
\begin{align}\label{uvam}
e_i^\pm(z)=\sum_{m\in \Z}(t_1^m\ot e_{i(m)}^\pm)z^{-m-1}=Y_W(e_i^\pm,z),\quad i\in I.
\end{align}
On the other hand,   any restricted  module $W$ for the Lie algebra $\wh\fm[\mu]$  is naturally a quasi-module
for the $\Gamma$-vertex  algebra $V(\wh\fm)$
  with $Y_W(e_i^\pm,z)=e_i^\pm(z)$ for $i\in I$.
\end{prpt}

We continue  the proof of Theorem \ref{intromain2} by assuming that Proposition \ref{universalva} holds.

\begin{lemt}\label{lem:gammavahom}For any $\fm=\fg, \fn_+$ or $\fn_-$,
the $\wh\fm$-module homomorphism $\varphi_{\fm,W}:V(\wh\fm)\rightarrow \< U_{\fm,W}\>_\Gamma$ is a $\Gamma$-vertex algebra homomorphism
\end{lemt}
\begin{proof} By the general principle given in Proposition \ref{gammavahom}, one only need to prove that
for $a\in T_\fm$ and $v\in V(\wh\fm)$,
\begin{align}
\label{gammahom3}\varphi_{\fm,W}(Y(a,w)v)&=\CY(\varphi_{\fm,W}(a),w)\,\varphi_{\fm,W}(v),
\end{align} and that for $a\in T_\fm$, $k\in \Z_N$ and $v\in V(\wh\fm)$,
\begin{align}
\label{gammahom4}
\varphi_{\fm,W}(R_{\mu^k}(Y(a,w)v))=\mathfrak R_{\mu^k}\(\CY(\varphi_{\fm,W}(a),w)\,\varphi_{\fm,W}(v)\),
\end{align}
provided that
\begin{align}\label{gammahom5}\varphi_{\fm,W}(R_{\mu^k}(v))=\mathfrak R_{\mu^k}(\varphi_{\fm,W}(v)).
\end{align}

We first prove the identity \eqref{gammahom3}. Indeed, for $i\in I$ and $v\in V(\wh\fm)$, one has that
\begin{eqnarray*}
& &\varphi_{\fm,W}(Y(e_i^\pm,w)v)=\varphi_{\fm,W}(e_i^\pm(w).v)\qquad \te{by}\ \eqref{uvavo}\\
&=&e_i^\pm(w).\varphi_{\fm,W}(v)\quad\qquad \te{(as $\varphi_{\fm,W}$ is a module homomorphism)}\\
&=&\CY(x_i^\pm(z),w)\varphi_{\fm,W}(v)
=\CY(\varphi_{\fm,W}(e_i^\pm),w)\varphi_{\fm,W}(v).\quad \te{by}\ \eqref{UWmodact},\eqref{eixi}.
\end{eqnarray*}
Next we prove that the condition \eqref{gammahom5} implies the identity \eqref{gammahom4}.
For $i\in I$, $k\in \Z_N$ and $v\in V(\wh\fm)$, one has that
\begin{eqnarray*}
&&\varphi_{\fm,W}\circ R_{\mu^k}(Y(e_i^\pm,w)v)=\varphi_{\fm,W}(Y(R_{\mu^k}(e_i^\pm),w)R_{\mu^k}(v))\qquad \te{by}\ \eqref{Gammava}\\
&=&\xi^k\varphi_{\fm,W}(Y(e_{\mu^k(i)}^\pm,w)R_{\mu^k}(v))=\xi^k\varphi_{\fm,W}(e_{\mu^k(i)}^\pm(w)R_{\mu^k}(v))
\qquad \te{by}\ \eqref{uvar}, \eqref{uvavo}\\
&=&\xi^k e_{\mu^k(i)}^\pm(w).\varphi_{\fm,W}(R_{\mu^k}(v))=\xi^k\CY(x_{\mu^k(i)}^\pm(z),w)\varphi_{\fm,W}(R_{\mu^k}(v))\quad \te{by}\ \eqref{UWmodact} \\
&=&\CY(\mathfrak R_{\mu^k}(x_i^\pm(z)),w)\varphi_{\fm,W}(R_{\mu^k}(v))=\CY(\mathfrak R_{\mu^k}(x_i^\pm(z)),w)\mathfrak R_{\mu^k}(\varphi_{\fm,W}(v))
\ \te{by}\ \eqref{Ractiononx},\eqref{gammahom5}\\
&=&\mathfrak R_{\mu^k}\(\CY(x_i^\pm(z),w)\varphi_{\fm,W}(v)\)=\mathfrak R_{\mu^k}\((e_{i}^\pm(z)).\varphi_{\fm,W}(v)\)\quad \te{by}\ \eqref{Gammava},\eqref{UWmodact}\\
&=&\mathfrak R_{\mu^k}\circ \varphi_{\fm,W}(e_i^\pm(w).v)=\mathfrak R_{\mu^k}\circ \varphi_{\fm,W}(Y(e_i^\pm,w)v),\quad \te{by}\  \eqref{uvavo}
\end{eqnarray*} as desired.
\end{proof}

Now we are ready to complete the proof of Theorem \ref{intromain2}. Firstly, recall that the restricted $\mathcal D_{\bm{P}}(\fm,\mu)$-module
$W$
is naturally a quasi-$\< U_{\fm,W}\>_\Gamma$-module under the action \eqref{actionyW}.
Next, by Lemma \ref{lem:gammavahom}, via the $\Gamma$-vertex algebra homomorphism $\varphi_{\fm,W}$,
 the quasi-$\< U_{\fm,W}\>_\Gamma$-module $(W,Y_{\mathfrak W})$
becomes a quasi-$V(\wh\fm)$-module with $Y_W(v,z)=Y_{\mathfrak W}(\varphi_{\fm,W}(v),z)$, $v\in V(\wh\fm)$.
In particular, by \eqref{UWmodact}, one gets that
\begin{equation}\begin{split}\label{equiact2}Y_W(e_i^\pm,w)=Y_{\mathfrak W}(x_i^\pm(z),w),\quad i\in I.
\end{split}\end{equation}
Finally, in view of Proposition \ref{universalva}, the quasi-$V(\wh\fm)$-module $(W,Y_W)$ admits an $\wh\fm[\mu]$-module structure
with action \eqref{uvam}.
In summary, by the actions \eqref{actionyW},\eqref{equiact2} and \eqref{uvam}, we have proved that there is an $\wh\fm[\mu]$-module structure on $W$ with
\begin{align*}e_i^\pm(w)=Y_W(e_i^\pm,w)=Y_{\mathfrak W}(x_i^\pm(z),w)=x_i^\pm(w),\quad i\in I.
\end{align*}
This finishes the proof of Theorem \ref{intromain2} as  the set
\[\{t_1^m\ot e_{i(m)}^\pm\mid i\in I, m\in \Z\}\cap \wh\fg[\mu]\]
generates the Lie algebra $\wh\fm[\mu]$.
\subsection{Proof of Theorem \ref{intromain1}}In this subsection we give the proof of Theorem \ref{intromain1}.
For this purpose, we need to prove two simple results  on  restricted modules of $\Z$-graded Lie algebras.
\begin{lemt}\label{zgraded1} Let $\mathfrak p=\oplus_{n\in \Z}\mathfrak p_ n$ be a $\Z$-graded Lie algebra and let $a\in \mathfrak p$.
Assume that $a.w=0$ for any restricted $\mathfrak p$-module $W$ and any $w\in W$.  Then $a=0$.
\end{lemt}
\begin{proof} Let $a\in \mathfrak p$ be as in lemma and assume conversely that $a\ne 0$. Then one may take a sufficiently large
 integer $s$ such that $a\not\in \mathfrak p_{\ge s}$, where $\mathfrak p_{\ge s}=\oplus_{n\ge s}\mathfrak p_n$.
Consider the left $\U(\mathfrak p)$-module $\U(\mathfrak p)/\U(\mathfrak p) \mathfrak p_{\ge s}$ via the left multiplication action.
Note that this $\mathfrak p$-module is restricted and so we have
\[a\in  \{x\in \mathfrak p\mid x.(\U(\mathfrak p)/\U(\mathfrak p) \mathfrak p_{\ge s})=0\}=\mathfrak p\cap \U(\mathfrak p) \mathfrak p_{\ge s}.\]
However, one can conclude from the PBW basis theorem that
\begin{align*} \mathfrak p\cap U(\mathfrak p) \mathfrak p_{\ge s}= \mathfrak p_{\ge s},
\end{align*}
which implies $a\in \mathfrak p_{\ge s}$, a contradiction.
\end{proof}
\begin{lemt}\label{lem:injective}Let $\mathfrak p$ be a $\Z$-graded Lie algebra, and let $f:\mathfrak p\rightarrow \mathfrak q$
be a  homomorphism of Lie algebras. Assume that for any restricted $\mathfrak p$-module $W$, there is a $\mathfrak q$-module structure on it with
\begin{align}\label{rescon} a.w=f(a).w,\quad a\in \mathfrak p,\ w\in W.
\end{align}
Then $f$ is an injective map.
\end{lemt}
\begin{proof} If $a\in \ker f$, then for any restricted $\mathfrak p$-module $W$, the relation \eqref{rescon} gives that $a.w=0$ for all $w\in W$. Thus the assertion follows from Lemma \ref{zgraded1}.
\end{proof}

Now, recall from Lemma \ref{thetahom} that  we already have the following surjective Lie homomorphisms
 \[\theta_{\fm,\bm{P}}:\mathcal D_{\bm{P}}(\fm,\mu)\rightarrow \wh\fm[\mu],\quad \fm=\fg,\ \fn_+\ \te{and}\ \fn_-.\]
By combining  Theorem \ref{intromain2} with Lemma \ref{lem:injective}, one finds that these homomorphisms
are also injective.
 This completes the proof of Theorem \ref{intromain1}.

\subsection{Proof of Proposition \ref{reducedtoun}}\label{subsec:proof-prop-D(m,1)}
When $\fg$ is of not type $A_1^{(1)}$, Proposition  \ref{reducedtoun} was proved in \cite{E-PBW-qaff}.
For convenience of the readers, in this subsection we give a proof of Proposition \ref{reducedtoun} for any $\fg$.
So throughout this subsection we assume that $\mu=\mathrm{Id}$.

Set $\mathcal D(\fm)=\mathcal D_{\bm{p}}(\fm,\mathrm{Id})$ and
$\theta_{\fm}=\theta_{\fm,\bm{p}}$  for $\fm=\fg, \fn_+$ or $\fn_-$.
Then we need to prove that all  $\theta_{\fm}$ are isomorphisms.
Firstly, we have that

\begin{prpt} \label{mainun}The Lie homomorphism $\theta_{\fg}:\mathcal D(\fg)\rightarrow \wh\fg$ is an isomorphism.
\end{prpt}
\begin{proof} Notice  that, for any restricted $\mathcal D(\fg)$-module $W$, the relation
\begin{align}\label{xii1}
(z-w)[x_i^\pm(z), x_i^\pm(w)]=0,\quad i\in I,
\end{align}
implies that (see Lemma \ref{1-commutator})
\[[x_i^\pm(z), x_i^\pm(w)]=c_i^\pm(w)\delta^{(0)}(z,w),\]
for some $c_i^\pm(z)\in \CE(W)$. But, on the other hand, one has that
\[ [x_i^\pm(z), x_i^\pm(w)]=-[x_i^\pm(w),x_i^\pm(z)]=-c_i^\pm(z)\delta^{(0)}(w,z)
=-c_i^\pm(w)\delta^{(0)}(z,w).\]
This gives that $c_i^\pm(w)=0$ for $i\in I$ and so we have that (see Lemma \ref{zgraded1})
\begin{align}\label{xii2}
[x_i^\pm(z),x_i^\pm(w)]=0,\quad i\in I.
\end{align}

When $\fg$ is of untwisted affine type, by replacing the relation \eqref{xii1} with
 \eqref{xii2}, the presentation $\mathcal D(\fg)$ of $\wh\fg$ was proved in \cite{MRY}.
When $\fg$ is of other types, the proof given in \cite{MRY} is also valid. We left the
details to the interesting readers.
\end{proof}

Let $\wt{\mathcal D}(\fg)$
 be the algebra generated by the same generators of $\mathcal D(\fg)$ with the defining relations
$(H)$, $(HX\pm)$ and $(XX)$, and let
 $\wt{\mathcal D}(\fg)_+$, $\wt{\mathcal D}(\fg)_-$ and $\wt{\mathcal D}(\fg)_0$ be the subalgebras of $\wt{\mathcal D}(\fg)$ generated by $x_{i,m}^+$, $x_{i,m}^-$ and
$h_{i,m}, c$ ($i\in I, m\in \Z$), respectively.
The following result is standard.

\begin{lemt}\label{lem:un2} $\wt{\mathcal D}(\fg)=\wt{\mathcal D}(\fg)_+\oplus \wt{\mathcal D}(\fg)_0\oplus \wt{\mathcal D}(\fg)_-$ and the algebra $\wt{\mathcal D}(\fg)_+$ (resp. $\wt{\mathcal D}(\fg)_-$) are free generated by the elements $x_{i,m}^+$
(resp. $x_{i,m}^-$), $i\in I, m\in \Z$.
\end{lemt}

We denote by $\wh{\mathcal D}(\fg)$ the quotient algebra of $\wt{\mathcal D}(\fg)$ modulo the
relation $(X\pm)$, and denote by $\overline{\mathcal D}(\fg)$ the quotient algebra of $\wh{\mathcal D}(\fg)$ modulo
the relation $(AS\pm)$ ($\wh{\mathcal D}(\fg)\ne \overline{\mathcal D}(\fg)$ only if $\fg$ is of type $A_1^{(1)}$).
Then $\mathcal D(\fg)$ is the quotient algebra of $\overline{\mathcal D}(\fg)$ obtained by modulo the
relation $(DS\pm$)$_{\bm{p}}$. Similar to $\mathcal D(\fg)$, all the algebras $\wt{\mathcal D}(\fg)$,
$\wh{\mathcal D}(\fg)$ and $\overline{\mathcal D}(\fg)$ are natural $\Z$-graded.

\begin{lemt}\label{lem:un3}
(i) In $\wt{\mathcal D}(\fg)$, one has that
\begin{align}\label{vanunrela1}f_{ij}(z,w)\cdot[[x_i^\pm(z),x_j^\pm(w)], x_k^\mp(w')]=0,\quad i,j,k\in I.
\end{align}
(ii) If $\fg$ is of type $A_1^{(1)}$, then in $\wh{\mathcal D}(\fg)$ one has that
\begin{align}\label{vanunrela2}
(z_1-z_2)[[x_i^\pm(z_1),[x_i^\pm(z_2),x_j^\pm(w)]],x_k^\mp(w')]=0,\quad
(i,j)\in\mathbb I,\, k\in I.
\end{align}
(iii) In $\overline{\mathcal D}(\fg)$, one has that
\begin{align}\label{vanunrela3} [[x_i^\pm(z_1),\cdots,[x_i^\pm(z_{1-a_{ij}}),x_j^\pm(w)]], x_k^\mp(w')]=0,\quad (i,j)\in \mathbb I,\ k\in I.
\end{align}
\end{lemt}

\begin{proof}
By applying the relations $(XX)$ and $(HX\pm$) in $\wt{\mathcal D}(\fg)$, one gets that
\begin{align*}&[[x_i^\pm(z),x_j^\pm(w)], x_k^\mp(w')]\\
=\ &\pm\delta_{i,k} [h_i(z), x_j^\pm(w)]\delta^{(0)}(z,w')\pm \delta_{j,k}[x_i^\pm(z), h_j(w)]\delta^{(0)}(w,w')\\
=\ &\pm(\delta_{i,k}a_{ij}x_j^\pm(w)-\delta_{j,k}a_{ji} x_i^\pm(z))\delta^{(0)}(z,w)\delta^{(0)}(w,w').
\end{align*}
Then the identity \eqref{vanunrela1} is implied by the fact that $\delta_{i,k}a_{ij}x_j^\pm(w)-\delta_{j,k} a_{ji} x_i^\pm(w)=0$ if $a_{ij}\ge 0$, and that
$(z-w)\,\delta^{(0)}(z,w)=0$ if $a_{ij}<0$.

Note that a similar argument of \eqref{xii2} shows that  for all $i\in I$, the relations
$[x_i^\pm(z),x_i^\pm(w)]=0$
hold in $\wh{\mathcal D}(\fg)$ and $\overline{\mathcal D}(\fg)$.
Therefore, in the rest of the proof, we can and do assume that
\begin{align}\label{eq:assume-f-ii}
f_{ii}(z,w)=1,\quad\te{for }i\in I
\end{align}
in the relation $(X\pm)$.

For the identities \eqref{vanunrela2} and \eqref{vanunrela3}, we only need to check the case that $k=i$ or $k=j$.
If $k=j$, then the identities \eqref{vanunrela2} and \eqref{vanunrela3} follow from the relation $(XX)$ with $i=j$.
So assume now that $k=i$. Let $W$ be an arbitrary restricted $\wh{\mathcal D}(\fg)$-module.
In view of the relation $(X\pm)$ in $\wh{\mathcal D}(\fg)$ and the equation \eqref{eq:assume-f-ii}, one finds  that as operators on $W$, the formal series
\begin{align*}:x_{i_1}^\pm(z_1)\cdots x_{i_a}^\pm(z_a):=\prod_{1\le s<t\le a} f_{i_s,i_t}(z_s,z_t)\cdot x_{i_1}^\pm(z_1)\cdots x_{i_a}^\pm(z_a),\
i_1,\cdots,i_a\in I,
\end{align*}
lie in the space $\mathrm{Hom}(W,W((z_1,\cdots,z_a)))$.

We first consider the case that $\fg$ is not of type $A_1^{(1)}$ and so $\wh{\mathcal D}(\fg)=\overline{\mathcal D}(\fg)$.
By definition, we have that
\begin{align*}
&[x_i^\pm(z), x_j^\pm(w)]=:x_i^\pm(z)x_j^\pm(w):((z-w)^{-1}-(w-z)^{-1})\\
=\ &:x_i^\pm(z)x_j^\pm(w):\delta^{(0)}(z,w)
=:x_i^\pm(w)x_j^\pm(w):\delta^{(0)}(z,w).
\end{align*}
More generally, one can easily check that for any $a\ge 1$,
\begin{align*}
\left[x_i^\pm(z),:\belowit{\underbrace{x_i^\pm(w)\cdots x_i^\pm(w)}}{(a-1)\te{-copies}}x_j^\pm(w):\right]
=:\belowit{\underbrace{x_i^\pm(w)\cdots x_i^\pm(w)}}{a\te{-copies}}x_j^\pm(w):\delta^{(0)}(z,w).
\end{align*}
This implies that for any $a\ge 1$,
\begin{equation}\begin{split}\label{norord}
&[x_i^\pm(z_1),\cdots,[x_i^\pm(z_a),x_j^\pm(w)]]\\
=\ &:\belowit{\underbrace{x_i^\pm(w)\cdots x_i^\pm(w)}}{a\te{-copies}}x_j^\pm(w):\delta^{(\bm{0})}(z_1,\cdots,z_a,w).
\end{split}\end{equation}

For convenience, set $b=1-a_{ij}$.
Then, by applying the relations $(XX)$ and $(HX\pm$), we have that
\begin{align*}
&[[x_i^\pm(z_1),\cdots,[x_i^\pm(z_b),x_j^\pm(w)],x^\mp_i(w')]]\\
=\ &\pm \sum_{1\le a\le b}[x_i^\pm(z_1),\cdots,[x_i^\pm(z_{a-1}),[h_i(w'),[x_i^\pm(z_{a+1}),\cdots,\\
& \cdot [x_i^\pm(z_b),x_j^\pm(w)]]]]]
\delta^{(0)}(z_a,w')\\
=\ &2\sum_{1\le a<a'\le b}X(z_1,\cdots,\hat{z_{a}},\cdots,z_b,w)\delta^{(0)}(z_a,z_{a'})\delta^{(0)}(z_a,w')\\
\ &+a_{ij}\sum_{1\le a\le b}X(z_1,\cdots,\hat{z_{a}},\cdots,z_b,w)\delta^{(0)}(z_a,w)\delta^{(0)}(z_a,w').
\end{align*}
where for any $a=1,\cdots,b$,
\[X(z_1,\cdots,\hat{z_{a}},\cdots,z_b,w)
=[x_i^\pm(z_1),\cdots,[x_i^\pm(z_{a-1}),[x_i^\pm(z_{a+1}),\cdots,[x_i^\pm(z_b),x_j^\pm(w)]]]].\]
Moreover, it follows from  \eqref{norord}  that for  $a,a'=1,\cdots,b$ with $a< a'$,
\begin{align*}
&X(z_1,\cdots,\hat{z_{a}},\cdots,z_b,w)\delta^{(0)}(z_a,z_a')\delta^{(0)}(z_a,w')\\
=\, &X(z_1,\cdots,\hat{z_{a}},\cdots,z_b,w)\delta^{(0)}(z_a,w)\delta^{(0)}(z_a,w')\\
=\, & :\belowit{\underbrace{x_i^\pm(w)\cdots x_i^\pm(w)}}{-a_{ij}\te{-copies}}x_j^\pm(w):\delta^{(\bm{0})}(z_1,\cdots,z_a,w)\delta^{(0)}(w,w').
\end{align*}
Thus, by combining these facts, we obtain that the formal series
\begin{align*}
&[[x_i^\pm(z_1),\cdots,[x_i^\pm(z_b),x_j^\pm(w)],x^\mp_i(w')]]\\
=\,&(b(b-1)+a_{ij}b):\belowit{\underbrace{x_i^\pm(w)\cdots x_i^\pm(w)}}{-a_{ij}\te{-copies}}x_j^\pm(w):\delta^{(\bm{0})}(z_1,\cdots,z_a,w)
\delta^{(0)}(w,w')\\
=\,& 0
\end{align*}
on any restricted $\wh{\mathcal D}(\fg)$-module.
This together with Lemma \ref{zgraded1} proves the identity  \eqref{vanunrela3} with $k=i$, as required.

Now, we turn to consider the case that $\fg$ is of type $A_1^{(1)}$.
Then by definition we have that
\begin{align*}
  &[x_i^\pm(z),x_j^\pm(w)]=:x_i^\pm(z)x_j^\pm(w):\delta^{(1)}(z,w)
  +\(\frac{\partial}{\partial z}:x_i^\pm(z)x_j^\pm(w):\)|_{z=w}\delta^{(0)}(z,w).
\end{align*}
From the relation $(X\pm)$ and the equation \eqref{eq:assume-f-ii}, we also have
\begin{align*}
  (z-w)^2x_i^\pm(z)\(\frac{\partial^{n_1+\cdots+n_k}}{\partial z_1^{n_1}\cdots \partial z_k^{n_k}}:x_i^\pm(z_1)\cdots x_i^\pm(z_k)x_j^\pm(w):\)|_{z_1=\cdots=z_k=w}
\end{align*}
lies in the space $\te{Hom}(W,W((z,w)))$
for any $k\ge 0$ and $n_1,\dots,n_k=0,1$.
This implies that
\begin{equation}\label{eq:A11-temp1}
\begin{split}
  &[x_i^\pm(z_1),[x_i^\pm(z_2),x_j^\pm(w)]]
  =:x_i^\pm(w)x_i^\pm(w)x_j^\pm(w):\delta^{(1,1)}(z_1,z_2,w)\\
  &\qquad+\(\frac{\partial}{\partial z_1}:x_i^\pm(z_1)x_i^\pm(w)x_j^\pm(w):\)|_{z_1=w}\delta^{(0,1)}(z_1,z_2,w)\\
  &\qquad+\(\frac{\partial}{\partial z_2}:x_i^\pm(w)x_i^\pm(z_2)x_j^\pm(w):\)|_{z_2=w}\delta^{(1,0)}(z_1,z_2,w)\\
  &\qquad+\(\frac{\partial^2}{\partial z_1\partial z_2}
  :x_i^\pm(z_1)x_i^\pm(z_2)x_j^\pm(w):\)|_{z_1=z_2=w}\delta^{(0,0)}(z_1,z_2,w).
  \end{split}
\end{equation}
By using the relations $(XX)$ and $(HX\pm)$, we get that
\begin{equation}\label{eq:A11-temp2}
\begin{split}
&[x_i^\mp(w'),[x_i^\pm(z_1),[x_i^\pm(z_2),x_j^\pm(w)]]]\\
=&
- 2:x_i^\pm(w)x_j^\pm(w):\delta^{(1,0,0)}(w',z_1,z_2,w)\\
&+ 2
\(\frac{\partial}{\partial z}:x_i^\pm(z)x_j^\pm(w):\)|_{z=w}\delta^{(0,0,0)}(w',z_1,z_2,w),
\end{split}
\end{equation}
and
\begin{equation}\label{eq:A11-temp3}
\begin{split}
&[x_i^\mp(w'),[x_i^\pm(z_1),[x_i^\pm(z_2),[x_i^\pm(z_3),x_j^\pm(w)]]]]\\
=&6\(\frac{\partial}{\partial z}:x_i^\pm(w)x_i^\pm(z)x_j^\pm(w):\)|_{z=w}
\delta^{({\bm 1}_1)}(w',z_1,z_2,z_3,w)\\
&+2\sum_{s=2}^4\(\frac{\partial}{\partial z}:x_i^\pm(w)x_i^\pm(z)x_j^\pm(w):\)|_{z=w}
\delta^{({\bm 1}_s)}(w',z_1,z_2,z_3,w)\\
&+2\sum_{1\le s<t\le 4}:x_i^\pm(w)x_i^\pm(w)x_j^\pm(w):
\delta^{({\bm 1}_{s,t})}(w',z_1,z_2,z_3,w),
\end{split}
\end{equation}
where ${\bm 1}_s=(\delta_{1s},\dots,\delta_{4s})$
and ${\bm 1}_{s,t}=(\delta_{1s}+\delta_{1t},\dots,\delta_{4s}+\delta_{4t})$.
Then the identity \eqref{vanunrela2} is implied by Lemma \ref{zgraded1}, the relation \eqref{eq:A11-temp2} and the fact that $(z-w)\delta^{(0)}(z-w)=0$.

For the identity \eqref{vanunrela3}, one can conclude from
 \eqref{eq:A11-temp1} that
\begin{equation}\label{eq:A11-temp4}
\begin{split}
 &:x_i^\pm(w)x_i^\pm(w)x_j^\pm(w):=0=\(\frac{\partial}{\partial z_1}:x_i^\pm(z_1)x_i^\pm(w)x_j^\pm(w):\)|_{z_1=w}\\
&\qquad =\(\frac{\partial}{\partial z_2}:x_i^\pm(w)x_i^\pm(z_2)x_j^\pm(w):\)|_{z_2=w}
\end{split}
\end{equation}
on any restricted $\bar{\mathcal D}(\fg)$-module.
Then it is easy to see that
 the identity \eqref{vanunrela3} is implied by Lemma \ref{zgraded1},
the relation \eqref{eq:A11-temp3} and the relation \ref{eq:A11-temp4} with $k=i$,
as required.
\end{proof}

Let $\mathcal D(\fg)_+$ and $\mathcal D(\fg)_-$ denote respectively
the subalgebras of $\mathcal D(\fg)$ generated by the elements $x_{i,m}^+$ and $x_{i,m}^-$ for $i\in I, m\in \Z$.
By combining  Lemma \ref{lem:un2} with Lemma \ref{lem:un3}, one immediately gets that
\begin{lemt}\label{lem:un4}  The algebra $\mathcal D(\fg)_+$ (resp. $\mathcal D(\fg)_-$) is isomorphic to the Lie algebra abstractly generated by the elements $x_{i,m}^+$ (resp. $x_{i,m}^-$), $i\in I, m\in \Z$ and subject to the relations
$(X+), (AS+), (DS+)_{\bm{p}}$ (resp. $(X-), (AS-), (DS-)_{\bm{p}}$).
\end{lemt}

Finally, from Proposition \ref{mainun}  and Lemma \ref{lem:un4},
it follows that both $\theta_{\fn_+}$ and $\theta_{\fn_-}$ are isomorphisms.
This completes the proof of Proposition \ref{reducedtoun}.

\subsection{Proof of Proposition \ref{universalva}}\label{subsec:proof-prop-universalva}
In this subsection, based on the theory of $\Gamma$-conformal Lie algebras developed
in \S\,\ref{subsec:Gamma-VA}, we give
a proof of Proposition \ref{universalva}.
We first consider the conformal Lie algebra $C_\fg$ associated to $\wh\fg$.
As a vector space
\[C_\fg=\(\C[T]\ot \underline{\fg}\)\oplus \C\rk_1,\]
where
 $\underline{\fg}=\fg\oplus \sum_{m\in \Z}\C t_2^m\rk_1'\subset \wh\fg$.
 We also introduce two subspaces of $C_\fg$ as follows
\[ C_{\fn_+}=\C[T]\ot \underline{\fn}_+\quad \te{and}\quad C_{\fn_+}=\C[T]\ot \underline{\fn}_-,\]
where $\underline{\fn}_+=\fn_+\oplus \sum_{m\in \Z_+}\C t_2^m\rk_1'$ and $\underline{\fn}_-=\fn_-\oplus \sum_{m\in \Z_+}\C t_2^{-m}\rk_1'$.
Let $T$ be the operator on $C_\fg$  defined by
\begin{align}\label{defT}T(T^m\ot x)=T^{m+1}\ot x,\quad T(\rk_1)=0,\end{align}
where $m\in \N$ and $x\in \underline{\fg}$. We define
\[Y_-: C_\fg\rightarrow \mathrm{Hom}(C_\fg, z^{-1}C_\fg[z^{-1}]),\quad x\mapsto \sum_{n\in \N} x_{(n)}z^{-n-1}\]
to be the unique linear map such that the property \eqref{proT} holds and such that
the non-trivial $n$-products on
$\underline{\fg}\oplus \rk_1$ are  as follows:
for any $\al, \beta\in \Delta$, $x\in \fg_\al$ and $y\in \fg_\beta$, if $\al+\beta\in \Delta^\times\cup \{0\}$, then
\begin{align}\label{npro1}
x_{(0)} y=[x,y],\quad x_{(1)} y=\<x,y\>\rk_1;\end{align}
if $\fg$ is of affine type, $x=t_2^{m_2}\ot \dot{x}$, $y=t_2^{n_2}\ot \dot{y}$ and $\al+\beta\in \Delta^0\setminus\{0\}$,  then
\begin{align}\label{npro2}
x_{(0)} y=[x,y]+\<\dot{x},\dot{y}\>m_2(T\ot t_2^{m_2+n_2}\rk_1'),\ x_{(1)}y=(m_2+n_2)\<\dot{x},\dot{y}\>t_2^{m_2+n_2}\rk_1'.
\end{align}

\begin{lemt}\label{lem:wh-C-Lie-iso}
The triple $(C_\fg,T,Y_-)$ defined above is a conformal Lie algebra, and $C_{\fn_+}$, $C_{\fn_-}$ are two conformal Lie subalgebras of it.
Moreover, the linear map $i_\fg: \wh{C}_\fg\rightarrow \wh\fg$ defined by
 \begin{align}\label{cfgiso}
 x(m)\mapsto t_1^m\ot x,\quad t_2^{n}\rk_1'(m)\mapsto t_1^{m+1}t_2^{n}\rk_1',\quad \rk_1(m)\mapsto \delta_{m,-1}\rk_1\end{align}
for $x\in \fg$ and $m,n\in \Z$, is an isomorphism of Lie algebras.
Similarly, the linear maps $i_{\fn_\pm}=i_\fg|_{\wh{C}_{\fn_\pm}}:\wh{C}_{\fn_\pm}\rightarrow \wh\fn_\pm$ are isomorphisms of Lie algebras.
\end{lemt}
\begin{proof} Using \eqref{defT}, it is easy to see that all the maps $i_\fg$, $i_{\fn_+}$ and $i_{\fn_-}$ are isomorphisms of vector spaces.
Then $\wh{C}_\fg$ admits a Lie algebra structure transferring from $\wh\fg$. In view of Lemma \ref{lem:relationwhfg},
the Lie bracket on $\wh{C}_\fg$ is given by \eqref{LieC} with the $n$-products defined by \eqref{npro1} and \eqref{npro2}.
Thus, by Lemma \ref{lem:determineC}, $(C_\fg,T,Y_-)$ is a conformal Lie algebra.
Moreover, it is obvious that $C_{\fn_+}$ and $C_{\fn_-}$  are invariant under the $n$-products \eqref{npro1} and \eqref{npro2}.
This implies that they are conformal Lie subalgebras of $C_{\fg}$ and the associated Lie algebras are isomorphic to $\wh\fn_\pm$.
\end{proof}

Recall the subset $T_\fm\subset C_\fm$ defined in \eqref{defTfm}.
Using Lemma \ref{C0genVC}, one gets that
\begin{lemt}\label{TmgeV} For any $\fm=\fg,\fn_+$ or $\fn_-$, $T_{\fm}$ is a generating subset of the universal vertex algebra $V_{C_\fm}$ associated to $C_\fm$.
\end{lemt}

We now define a linear transformation $R_\mu$ on  $C_\fg$  such that
\begin{align*}
R_\mu(T^n\ot x)=\xi^{n+1}\,T^n\ot \mu(x),\quad R_\mu(\rk_1)=\rk_1
\end{align*}
for $n\in \N$, $x\in \fg_\al, \al\in \Delta^\times \cup \{0\}$,
and such that (the following actions exist only when $\fg$ is of affine type)
\begin{align*}
R_\mu(T^n\ot h)&=\xi^{n+1}(T^n\ot \mu(h)+\rho_\mu(\dot{h})\,T^{n+1}\ot t_2^{m_2}\rk_1'),\\
 R_\mu(T^n\ot t_2^{n_2}\rk_1')&=
\xi^{n}\,T^n\ot t_2^{n_2}\rk_1',
\end{align*}
for $n\in \N$, $h=t_2^{m_2}\ot \dot{h}$ with $m_2\ne 0$, $\dot h\in\dfh_{[m_2]} $ and $n_2\in \Z$.
\begin{lemt}\label{defgammaaction}
The map $R_\mu$ has order $N$ and satisfies the axiom $T R_\mu=\xi^{-1}R_\mu T$.
\end{lemt}
\begin{proof} The assertion is obvious when $\fg$ is of finite type.
For the case that  $\fg$ is of affine type, recall the automorphism $\dot\mu$ on $\dg$ and the linear functional
$\rho_\mu$ on $\dfh$ introduced in \S\,2.2.
Then one can conclude from the fact $\mu^N=1$ and
\cite[Lemma 2.1 (b),(c)]{CJKT-uce} that
\begin{align*}
  \sum_{k\in \Z_N}\rho_\mu(\dot\mu^k(\dot h))=0,\quad \forall\ \dot h\in \dfh.
\end{align*}
Using this and the action \eqref{actmu2}, one can easily check that
$R_\mu^N(T^n\ot h)=T^n\ot h$ for
$n\in \N$, $h=t_2^{m_2}\ot \dot{h}$ with $m_2\ne 0$ and $\dot h\in\dfh_{[m_2]}$. This  implies the first assertion in lemma and the
second one is obvious.
\end{proof}

In view of Lemma \ref{defgammaaction}, we now have the following group homomorphism
\[R:\Gamma\rightarrow \mathrm{GL}(C_\fg),\quad \mu^n\mapsto R_{\mu^n}=\(R_\mu\)^n,\quad n\in \Z_N,\]
which satisfies the condition that
\begin{align}\label{comgammaaction} T R_{\mu^n}=\xi^{-n}R_{\mu^n} T,\quad n\in \Z_N.
\end{align}

\begin{lemt}\label{gammaconformalg}  $(C_\fg,T,Y_-,R)$ is a $\Gamma$-conformal Lie algebra and $C_{\fn_+}$, $C_{\fn_-}$
are $\Gamma$-conformal Lie subalgebras of $C_\fg$.
Moreover, the associated Lie algebra  $(\wh{C}_\fm)_\Gamma$ is isomorphic to the Lie algebra $\wh\fm[\mu]$  for $\fm=\fg,\fn_+$ or $\fn_-$.
 \end{lemt}
\begin{proof} Recall  the linear map $\wh{R}_\mu: \wh{C}_\fg\rightarrow \wh{C}_\fg$ induced by $R_\mu$ defined
in \eqref{defwhRg}. By using the explicit action of $\wh\mu$ on $\wh\fg$ given in Lemma
\ref{defhatmu}, one can check that $\wh{R}_\mu=i_\fg\circ\wh{\mu}\circ i_\fg^{-1}$. So $\wh{R}_\mu$ is
an automorphism of $\wh{C}_\fg$. This together with Lemma \ref{whRg} and \eqref{comgammaaction} proves
that  $(C_\fg,T,Y_-,R)$ is a $\Gamma$-conformal Lie algebra.
Note that both  $C_{\fn_+}$ and $C_{\fn_-}$  are stable under the map $R_\mu$. So they are $\Gamma$-conformal Lie
subalgebras of $C_\fg$, which completes the proof of the first assertion.
The second assertion follows from Remark \ref{whCGamma} and the moreover statement in Lemma \ref{lem:wh-C-Lie-iso}.
\end{proof}

Combing Lemma \ref{gammaconformalg} with Lemma \ref{gammacontova}, one knows that $V_{C_\fm}$ admits a natural  $\Gamma$-vertex algebra structure.
Recall the subalgebra $\wh\fm^{(+)}$ of $\fm$ defined in \S\,\ref{subsec:proof-thm-intromain2}
and the subalgebra $\wh{C}_\fm^{(+)}$ of $\wh{C}_\fm$ defined in \S\,5.3.

\begin{lemt}\label{charwhfm+} For any $\fm=\fg,\fn_+$ or $\fn_-$, one has that $\wh\fm^{(+)}=i_\fm(\wh{C}_\fm^{(+)})$.
\end{lemt}
\begin{proof} It is easy to see that
\[i_\fm(\wh{C}_\fm^{(+)})=\te{Span}_\C\set{t_1^m\ot x,\, t_1^{m+1}t_2^n\rk_1' }{x\in \fg,n\in\Z,m\in\N}\cap \wh\fm\]
is the subalgebra of $\wh\fm$ generated by
by the elements in \eqref{genm+}, and hence coincides with $\wh\fm^{(+)}$, as required.
\end{proof}

Now we are ready to finish the proof of Proposition \ref{universalva}.
In view of Lemma \ref{charwhfm+}, there is a natural $\Gamma$-vertex algebra structure on $V(\wh\fm)$
  transferring from $V_{C_\fm}$,
with $T_\fm$ as a generating subset (Lemma \ref{TmgeV}).
This proves the first assertion of Proposition \ref{universalva}.
The second assertion of Proposition \ref{universalva} is implied by Lemma \ref{gammaconformalg} and
Proposition \ref{quasitores}, as desired.



\end{document}